\documentclass[a4paper, 13pt]{article}
\usepackage[english]{babel}

\usepackage[T1]{fontenc}
\usepackage{lmodern}   
\usepackage{xcolor}

\usepackage{amsmath,amsthm,amssymb,amsfonts}
\usepackage{mathtools}
\usepackage{tensor}
\usepackage{tikz-cd}
\usepackage{esint}
\usepackage{mathrsfs}
\usepackage{thmtools}
\usepackage{slashed}

\usepackage{bbm}
\usepackage{microtype}
\usepackage{accents}

\usepackage{enumitem}
\usepackage[toc,page]{appendix}
\usepackage{imakeidx}

\usepackage{crossreftools}
\usepackage{hyperref}
\usepackage{cleveref}

\usepackage{emptypage}
\usepackage{parskip}
\usepackage[a4paper,	bindingoffset=0in,	left=0.8in,	right=0.8in,	top=1in,	bottom=1in,	footskip=.25in]{geometry} 
\usepackage[backend=biber,style=alphabetic,sorting=ynt,sortcites]{biblatex}
\allowdisplaybreaks[1]

\theoremstyle{plain}
\newtheorem{theorem}{Theorem}[section]
\newtheorem*{theorem*}{Theorem}
\newtheorem{thmx}{Theorem}

\newtheorem{cor}[theorem]{Corollary}
\newtheorem{lem}[theorem]{Lemma}
\newtheorem{prop}[theorem]{Proposition}
\theoremstyle{definition}
\newtheorem{ex}[theorem]{Example}
\newtheorem{exs}[theorem]{Examples}
\newtheorem{dfn}[theorem]{Definition}
\newtheorem{rem}[theorem]{Remark}
\newtheorem{rems}[theorem]{Remarks}

\theoremstyle{remark}

\newcommand{\RNum}[1]{\uppercase\expandafter{\romannumeral #1\relax}}

\makeatletter
\providecommand*{\twoheadrightarrowfill@}{%
  \arrowfill@\relbar\relbar\twoheadrightarrow
}
\providecommand*{\twoheadleftarrowfill@}{%
  \arrowfill@\twoheadleftarrow\relbar\relbar
}
\providecommand*{\xtwoheadrightarrow}[2][]{%
  \ext@arrow 0579\twoheadrightarrowfill@{#1}{#2}%
}
\providecommand*{\xtwoheadleftarrow}[2][]{%
  \ext@arrow 5097\twoheadleftarrowfill@{#1}{#2}%
}
\newcommand\setItemnumber[1]{\setcounter{enum\romannumeral\@enumdepth}{\numexpr#1-1\relax}}
\makeatother

\newcommand\norm[1]{\left\lVert#1\right\rVert}

\DeclareMathOperator\supp{supp}

\DeclareMathOperator\ev{ev}

\DeclareMathOperator\reg{reg}
\DeclareMathOperator\areg{areg}

\DeclareMathOperator\blup{blup}

\DeclareMathOperator\hblup{\mathcal{H}\!\blup}

\DeclareMathOperator\dom{Dom}

\newcommand{\R}{\mathbb{R}}

\newcommand{\N}{\mathbb{N}}

\newcommand{\Z}{\mathbb{Z}}

\newcommand{\cA}{\mathcal{A}}

\newcommand{\cC}{\mathcal{C}}

\newcommand{\cF}{\mathcal{F}}

\newcommand{\cH}{\mathcal{H}}

\newcommand{\cL}{\mathcal{L}}

\newcommand{\cQ}{\mathcal{Q}}

\newcommand{\cT}{\mathcal{T}}

\newcommand{\cW}{\mathcal{W}}

\AtEveryBibitem{\clearfield{url}}

\newcommand{\Gr}{\mathcal{G}\mathrm{r}}

\newcommand{\tF}{\mathfrak{t}\cF}

\newcommand{\GR}{\mathrm{Grass}}

\newcommand{\chblup}{\cC\!\cH\!\blup}
\newcommand{\cblup}{\mathcal{C}\!\blup}
\begin{document}
\title{Blow-up groupoid of singular foliations}
\author{Omar Mohsen}
\date{}
\maketitle
\begin{abstract} 
	We introduce a blow-up construction of a smooth manifold along the singular leaves of an arbitrary singular foliation in the sense of Stefan and Sussmann, as well as a blow-up construction of the holonomy groupoid defined by Androulidakis and Skandalis. Our construction gives a locally compact locally Hausdorff groupoid, which can be regarded as a desingularisation of the singular foliation. We show that it retains some smooth structure.
 \end{abstract}
\setcounter{tocdepth}{2}
\tableofcontents
\section*{Introduction}

A singular foliation on a manifold $M$ is a submodule $\cF$ of $\Gamma_c(TM)$ of the $C^\infty(M)$-module of compactly supported vector fields such that $\cF$ is locally finitely generated and is closed under taking Lie brackets. We call the foliation regular if $\cF$ is the module of sections of a subbundle of $TM$. By a theorem of Stefan and Sussmann theorem \cite{Sussmann,Stefan}, a singular foliation gives a partition of $M$ into connected immersed submanifolds, so called leaves. The foliation is regular if and only if all leaves are of the same dimension.

For regular foliations, Ehresmann \cite{EhresmannFoliation} and Winkelnkemper \cite{WinkenkemperFoliation} defined the holonomy groupoid, see also \cite{PradinesFoliation,PradinesFoliation2}. This construction was generalised by Debord \cite{DebordFoliation2001} to some singular foliations and then by Androulidakis and Skandalis \cite{AS1} to arbitrary singular foliations. We will denote the holonomy groupoid by $\cH(\cF)\rightrightarrows M$. The holonomy groupoid plays a fundamental role in noncommutative geometry, see \cite{ConnesSurvey,ConnesSkandalis,ConnesTransverse,ConnesBook}. Nevertheless a major issue is that, for singular foliations, $\cH(\cF)$ is neither locally compact nor locally Hausdorff. It is well known in the literature that non Hausdorff locally compact groupoids are more complicated to study and they fail to satisfy many properties that are satisfied by locally compact Hausdorff groupoids. It is even less clear which results remain true for non locally compact non Hausdorff groupoids.

In this article, we construct a locally compact locally Hausdorff blowup groupoid $$\hblup(\cF)\rightrightarrows\blup(\cF)$$ of the holonomy groupoid of singular foliation. Notice that we blowup both the space of objects and the space of arrows. If $\cF$ is regular, then $\blup(\cF)=M$ and $\hblup(\cF)=\cH(\cF)$. Before we proceed with its properties we give an example to illustrate the construction \begin{ex}Consider $SL_n(\R)$ acting on $\R^n$. This gives a singular foliation whose leaves are the orbits. The holonomy groupoid is equal to $$\cH(\cF)=\big(\R^n\backslash\{0\}\times \R^n\backslash\{0\}\big)\sqcup SL_n(\R).$$ The groupoid  structure is the disjoint union of the pair groupoid on $\R^n\backslash\{0\}\times \R^n\backslash\{0\}$ and the group structure on $SL_n(\R)$. The topology is given by \begin{itemize}
\item $\R^n\backslash\{0\}\times \R^n\backslash\{0\}$ is an open subset with its usual topology
\item $SL_n(\R)$ is a closed subset with its usual topology
\item a net $(x_\alpha,y_\alpha)\in \R^n\backslash\{0\}\times \R^n\backslash\{0\}$ converges to $g\in SL_n(\R)$ if and only if there exists $g_\alpha\in SL_n$ such that $g_\alpha x_\alpha =y_\alpha $ and $g_\alpha \to g$. Therefore the topology is neither locally compact nor locally Hausdorff. The idea of the blowup of $\cH(\cF)$ is that if a net $(x_\alpha,y_\alpha)$ converges in $\cH(\cF)$, then (up to a subnet) the limit set is the coset $gH_x$ for some $x\in S^{n-1}$, where $H_x=\{T\in SL_n:Tx=x\}$. The blowup construction adds the direction of convergence to the space of objects and then quotients $SL_n$ by the $H_x$ to obtain the space of arrows.
\end{itemize}
The blowup construction for this foliation is as follows. The space of objects  $$\blup(\cF)=\blup(\R^n,0)=\R^n\backslash\{0\}\sqcup \mathbb{P}^{n-1}$$ is the classical blowup of $\R^n$ at $0$. The blowup groupoid is $$\hblup(\cF)=\bigg(\R^n\backslash\{0\}\times \R^n\backslash\{0\}\bigg)\sqcup \big(\R^n\backslash\{0\}\times \R^n\backslash\{0\}\big)/\R^*\subseteq \blup(\R^n\times \R^n,(0,0)).$$ The inclusion means that $\hblup(\cF)$ is an open subset of $\blup(\R^n\times \R^n,(0,0))$ whose complement is $\mathbb{P}^{n-1}\times \{0\}\sqcup\{0\}\times \mathbb{P}^{n-1}$.
\end{ex} 
The main result of this article is that the above construction can be done for an arbitrary singular foliation. Furthermore the resulting groupoid retains some smoothness. The results of this article can be summarised in the following theorem.
\begin{thmx}\label{intro thm}
\begin{enumerate}
\item The space $\blup(\cF)$ is a topological blowup of $M$. In the sense that $\blup(\cF)$ is a separable metrizable locally compact space with a natural map $\pi:\blup(\cF)\to M$ which is proper and there exists $U\subseteq M$ an open dense subset such that $\pi:\pi^{-1}(U)\to U$ is a homeomorphism.
\item Over the space $\blup(\cF)$, we construct a vector bundle denoted $T\cF$. If $u\in \pi^{-1}(U)$, then the fiber of $T\cF$ at $u$ can be naturally  identified with the tangent space at $\pi(u)$ of the leaf of $\cF$. For a general $u\in \blup(\cF)$, $T\cF_u$ projects onto the tangent space of the leaf at $\pi(u)$.
\item We construct a separable locally metrizable locally compact groupoid $\hblup(\cF)\rightrightarrows \blup(\cF)$ such that $\hblup(\cF)$ is a continuous family groupoid in the sense of Paterson \cite{ContFamilyGroupoids} whose Lie algebroid is $T\cF$. In particular if $x\in \blup(\cF)$, then $\hblup(\cF)_x$ is a smooth manifold whose tangent space at $x$ is equal to $T\cF_x$.
\item We construct a $*$-homomorphism between $C^*\cF$ defined in \cite{AS1} and $C^*\hblup(\cF)$ $$\fint:C^*\cF\to C^*\hblup(\cF).$$
\end{enumerate}
\end{thmx}
In general the space $\blup(\cF)$ isn't a smooth manifold. Nevertheless, \begin{itemize}
\item the space $\blup(\cF)$ is defined as a closed subset of a smooth manifold and thus it inherits a sheaf of smooth functions.
\item The pullback of smooth functions on $M$ by $\pi$ is smooth. Furthermore $\pi:\pi^{-1}(U)\to U$ is a diffeomorphism.
\item if $\blup(\cF)$ is a smooth manifold, then $\hblup(\cF)$ is a Lie groupoid.
\end{itemize}
\paragraph{Applications.}Since the groupoid $\hblup(\cF)$ is locally compact. Many of the notions of noncommutative geometry are well defined for it. For example amenability \cite{AnaRenaultGrps}, Baum-Connes conjecture \cite{BaumConnesHigsonConj}.\footnote{Cf. \cite{AS4} where the Baum-Connes conjecture using $\cH(\cF)$ is only defined for a special class of singular foliations. Furthermore one cannot use results like \cite{TuBC} to get Baum-Connes for amenable singular foliations (which to our knowledge wasn't defined before our work).}

Since $T\cF$ is a vector bundle, one can define orientation, spin structure, characteristic classes of $\cF$ to be those of $T\cF$. Furthermore since the groupoid $\hblup(\cF)$ admits a Lie algebroid, one can construct a Dirac operator (under the assumption $T\cF$ being spin) and thus obtain a longitudinal Dirac operator for singular foliations, see \cite{ConnesTransverse} for applications in the case of regular foliations. We plan to investigate this further in a future paper.

We also expect our groupoid to give a simple treatment of the blowup groupoid of stratified manifolds defined in \cite{DebordJMKduality,DebordJMRochonPseudodiff}.

 Finally this blowup groupoid is used by the author with Androulidakis, van Erp and Yuncken \cite{MohsenMaxHypo} to settle the Helffer Nourrigat conjecture \cite{HelfferNourrigatBook}. It is also used by the author \cite{MohsenIndexmaxhypo} to compute the index of some maximally hypoelliptic differential operators . Both are based upon the notion of half-continuity of fields of $C^*$-algebras, which we introduce in Section \ref{sec:semicont fields}.
\paragraph{Structure of the article.}
We think it is beneficial to consider first singular foliations which come from actions of connected Lie groups, which is treated in Section \ref{sec:group act}. In Section \ref{sec:general case}, the general case is treated. The reader can go directly to Section \ref{sec:general case} if they wish to.
\section*{Acknowledgments}
This work was done partially while the author was a postdoc in University of Muenster and was funded by the Deutsche Forschungsgemeinschaft (DFG, German Research Foundation) - Project-ID 427320536 - SFB 1442, as well as Germany's Excellence Strategy EXC 2044 390685587, Mathematics M\"{u}nster: Dynamics-Geometry-Structure.

The author thanks G. Skandalis for many helpful and illuminating discussions. The author also thanks Androulidakis his help with the formulation of Theorem \ref{thm:AZ}.
\section{Blowup construction}\label{sec:group act}
In this article, a smooth manifold is by definition Hausdorff. We will explicitly say non Hausdorff manifolds, if that's the case. If $M$ is a smooth manifold, $X$ a vector field on $M$, $x\in M$, $t\in \R$, then $\exp_X^t(x)$ denotes the flow of $X$ at time $t$ starting from $x$, if defined. If $t=1$, then we write $\exp_X(x)$.
\subsection{Regular points and blow-up space}
\paragraph{Regular points.}
Let $G$ be a connected Lie group acting smoothly on a smooth manifold $M$. The derivative of the action is the anchor map $\natural:\mathfrak{g}\to \Gamma(TM)$, where $\mathfrak{g}$ is the Lie algebra of $G$. One has \begin{equation}\label{eqn:bracket}
 \natural([X,Y])=[\natural(X),\natural(Y)],\quad X,Y\in \mathfrak{g}.
\end{equation}
Let $x\in M$. We will denote by $\natural_x:\mathfrak{g}\to T_xM$ the evaluation of $\natural$ at $x$. Let $\mathfrak{h}_x=\ker(\natural_x)$. The space $\natural_x(\mathfrak{g})\simeq \frac{\mathfrak{g}}{\mathfrak{h}_x}$ is the tangent space of the orbit $G\cdot x$ at $x$.
 \begin{dfn} A point $x$ is called regular if the function  $y\mapsto \dim(\mathfrak{h}_y)$ is continuous at $x$. The set of regular points is denoted $M_{\reg}$.
\end{dfn}
By \cite[Prop 2.1]{AS1}, the set $M_{\reg}$ is an open dense set.
\paragraph{Blow-up space.}Let $$\GR(\mathfrak{g})=\sqcup_{0\leq k\leq \mathfrak{g}}\GR_k(\mathfrak{g})$$ where $\GR_k(\mathfrak{g})$ is the Grassmannian manifold of linear subspaces $\mathfrak{g}$ of dimension $k$. We equip $\GR(\mathfrak{g})$ with the disjoint union topology. Let \begin{align*}
 i:M_{\reg}\to  \GR(\mathfrak{g})\times M,\quad i(x)=(\mathfrak{h}_x,x).
\end{align*}
The map $i$ is a continuous embedding. The blow-up space is defined by $$\blup(G\ltimes M):=\overline{i(M_{\reg})}\subseteq \GR(\mathfrak{g})\times M.$$
It is convenient to also define for $x\in M$, $$\blup(G\ltimes M)_x  :=\{V\in \GR(\mathfrak{g}):(V,x)\in \blup(G\ltimes M)\}.$$ By compactness of Grassmannian spaces, $\blup(G\ltimes M)_x$ is a nonempty closed subset of $\GR(\mathfrak{g})$.
 \begin{prop} \label{prop:gamma prop}Let $x\in M$, $V\in \blup(G\ltimes M)_x$.
\begin{enumerate}
\item If $x\in M_{\reg}$, then $\blup(G\ltimes M)_x=\{\mathfrak{h}_x\}$.
\item One has $V\subseteq \mathfrak{h}_x$.
\item The vector subspace $V$ is a Lie subalgebra of $\mathfrak{g}$.
\item If $g\in G$, then $Ad(g)V\in \blup(G\ltimes M)_{g\cdot x}$.
\end{enumerate}
\end{prop}
\begin{proof}
1 and 2 are obvious. For 3, since $V$ is the limit of Lie subalgebras, $V$ is also a Lie subalgebra. For 4, let $V\in \blup(G\ltimes M)_x$, $x_n\in M_{\reg}$ such that $x_n\to x$ and $\mathfrak{h}_{x_n}\to V$. Since $Ad(g)\mathfrak{h}_{x_n}=\mathfrak{h}_{g \cdot x_n}$, and $g \cdot x_n\to g\cdot x$. One obtains that 
$ \lim_{n\to \infty}Ad(g)\mathfrak{h}_{x_n}=Ad(g)V \in \blup(G\ltimes M)_{g\cdot x}$.
\end{proof}
\begin{ex}Consider $G=\R$ acting on $M=\R$ by the flow of the vector field $X=\rho(x)\frac{\partial }{\partial x}$, where $\rho $ is any smooth function such $\rho>0$ on $]-1,1[$ and $\supp(\rho)=[-1,1]$. One easily sees that $$\blup(G\ltimes M)_x=\begin{cases}\{0\}&\text{if}\quad x\in]-1,1[\\\{\R\}&\text{if}\quad x\in ]-\infty,-1[\cup ]1,+\infty[\\\{0,\R\}&\text{if}\quad x\in\{-1,1\}.\end{cases}$$ It follows that $\blup(G\ltimes M)$ is equal to $]-\infty,-1]\sqcup [-1,1]\sqcup [1,+\infty[ $ with the three intervals disjoint by doubling the endpoints $-1$ and $1$.
\end{ex}
\subsection{Longitudinal smoothness}
If $V\subseteq \mathfrak{g}$ is a Lie subalgebra, then we denote by $\exp(V)$ the connected immersed Lie subgroup of $G$ whose Lie algebra is $V$. 
\begin{theorem}\label{closed} Let  $x_n\in M$, $V\in \GR(\mathfrak{g})$ such that $x_n\to x$ and $\mathfrak{h}_{x_n}\to V$ in $\GR(\mathfrak{g})$. Then $\exp(V)$ is a closed Lie subgroup of $G$.
\end{theorem}

\begin{rem}\begin{itemize}
\item Since $G/\exp(\mathfrak{h}_x)$ is a cover of $G\cdot x$, one can think of $G/\exp(V)$ as an orbit at $x$ with tangent space $\mathfrak{g}/V$. Thus Theorem \ref{closed} gives smoothness of these auxiliary orbits. 
\item The space $\blup(G\ltimes M)$ is defined by taking the closure of $\mathfrak{h}_x$ for $x\in M_{\reg}$, because we think it is more natural to only take the closure of regular points. Theorem \ref{closed} is true without supposing that $x_n\in M_{\reg}$.
\end{itemize}
\end{rem}
Theorem \ref{closed} is a corollary of the periodic bounding lemma. \begin{lem}[Periodic bounding lemma \cite{period1,period2}]\label{lem:periodic}Let $M$ be a smooth manifold, $X$ a vector field, $K\subseteq M$ a compact subset, then there exists $\epsilon>0$ such that for any $x\in K$ either $X(x)=0$ or $\exp^t_X(x)\neq x$ for all $0<|t|<\epsilon$.
\end{lem}
For the rest of this subsection, we equip $\mathfrak{g}$ with a Euclidean norm. The following lemma appears in \cite[Proposition 1.1]{Debord2013}.
 \begin{lem}[Parametrised periodic bounding lemma]\label{lemma 2}Let $K\subseteq M$ be a compact subset. Then there exists $\eta>0$ such that for every $y\in K$, $Y\in \mathfrak{g}$ if $\exp(Y)\cdot y=y$, then either $Y\in \mathfrak{h}_y$ or $\norm{Y}\geq \eta$.
  \end{lem}
  \begin{proof}Let $X$ be the vector field on $M\times \mathfrak{g}$ which at $(y,Y)$ is equal to $(\natural_y(Y),0)$. The flow of $X$ at $(y,Y)$ is equal to $(\exp(tY)\cdot y,Y)$. Let $\epsilon$ be obtained from Lemma \ref{lem:periodic} applied to $X$ and $K\times B$, where $B\subseteq \mathfrak{g}$ is the unit ball. If $Y\in \mathfrak{g}$ such that $Y\notin \mathfrak{h}_y$ and $\exp(Y)\cdot y=y$, then by Lemma \ref{lem:periodic}, $\exp_X^t(y,\frac{Y}{\norm{Y}})\neq (y,\frac{Y}{\norm{Y}})$ for all $0<|t|<\epsilon$. In other words $\exp(t Y)\cdot y\neq y$ for all $0<|t|<\frac{\epsilon}{\norm{Y}}$. Since $\exp(Y)\cdot y=y$, it follows that $\norm{Y}\geq \epsilon$.
  \end{proof}
Theorem \ref{closed}  easily follows from the following lemma.
 \begin{lem}\label{lemama 123}There exists $\eta''$ such that if $X\in \mathfrak{g}$ with $\norm{X}<\eta''$ and $\exp(X)\in \exp(V)$, then $X\in V$.
 \end{lem}
\begin{proof}Let $\eta$ be obtained from Lemma \ref{lemma 2} applied to $K=\{x_n:n\in \N\}\sqcup \{x\}$. Let $\eta'$ be such that $\exp:\mathfrak{g}\to G$ is an embedding when restricted to $B(0,\eta')$ the ball around $0$ of radius $\eta'$. We denote by $\log:\exp(B(0,\eta'))\to \mathfrak{g}$ the inverse of $\exp$. Let $X\in \mathfrak{g}$ be such that $\norm{X}<\min(\eta,\eta')$ and $\exp(X)\in \exp(V)$. Then there exists $X_1,\cdots,X_k\in V$ such that $\exp(X)=\exp(X_1)\cdots \exp(X_k)$. By approximating each $X_i$ with elements in $\mathfrak{h}_{x_n}$, we construct a sequence $g_n\in \exp(\mathfrak{h}_{x_n})$ such that $g_n\to \exp(X)$. Since $\exp(X)\in \exp(B(0,\eta'))$. It follows that for $n$ big enough $g_n\in \exp(B(0,\eta'))$. Hence $\log(g_n)$ is well defined and $\log(g_n)\to X$. Since $g_n\in \exp(\mathfrak{h}_{x_n})$, it follows that $g_n\cdot x_n=x_n$. By Lemma \ref{lemma 2}, $\log(g_n)\in \mathfrak{h}_{x_n}$ for $n$ big enough. Therefore $X\in V$.
\end{proof}
In the proof of Lemma \ref{lemma 2} by taking $K$ a compact neighbourhood of $x$, one obtains\begin{theorem}\label{closed2}Let $(V,x)\in \blup(G\ltimes M)$, $S\subseteq \mathfrak{g}$ a subspace transverse to $V$. Then there exists an open neighbourhood $S'$ of $0\in S$, and an open neighbourhood $L$ of $(V,x)\in \blup(G\ltimes M)$ such that for any $(W,y)\in L$, the map $$\exp:S'\to G/\exp(W),\quad X\to \exp(X)\exp(W)$$ is an embedding.
\end{theorem}
\subsection{Blow-up groupoid}\label{sec:blup groups action}
In this section we define a groupoid $\hblup(G\ltimes M)\rightrightarrows \blup(G\ltimes M)$. The set $\hblup(G\ltimes M)$ denotes the set of all left cosets of $\exp(V)$ for $(V,x)\in \blup(G\ltimes M)$. By Proposition \ref{prop:gamma prop}.d, the set of left cosets agree with the set of right cosets. Hence we define \begin{align*}
\hblup(G\ltimes M)=\{(g\exp(V),x):(V,x)\in \blup(G\ltimes M)\}
\end{align*}
The groupoid structure is given by 
\begin{align*}
&s(g\exp(V),x)=(V,x),\quad r(g\exp(V),x)=(Ad(g)V,gx)\\
&(g \exp(Ad(h)V),hx)\cdot (h\exp(V),x)=(gh\exp(V),x)\\
&(g\exp(V),x)^{-1}=(g^{-1}\exp(Ad(g)V),gx)
\end{align*}
The the product is well defined (independent of the choice of $g,h$ in their respective cosets) because $gh\exp(V)$ is equal to the product of  $g \exp(Ad(h)V)$ and $h\exp(V)$ as subsets of $G$. Same for the inverse.

The topology is defined as follows. One has a natural map $$\pi:G\times \blup(G\ltimes M)\to \hblup(G\ltimes M),\quad (g,(V,x))\to (g\exp(V),x).$$The topology on $\hblup(G\ltimes M)$ is the quotient topology. In other words, a sequence (or a net) $(g_n\exp(V_n),x_n)$ converges to $(g\exp(V),x)$ if \begin{itemize}
\item $x_n\to x$ and $V_n\to V\in \GR(\mathfrak{g})$
\item there exists $y_n\in \exp(V_n)$, $y\in \exp(V)$ such that $g_ny_n\to gy$.
\end{itemize}
This topology is metrizable. Let $d_M,d_G,d_{\mathfrak{g}}$ a metric on $M$, $G$, $\GR(\mathfrak{g})$ respectively. Then we define $$d((g\exp(V),x),(h\exp(W),y))=d_M(x,y)+d_{\mathfrak{g}}(V,W)+d_G(g\exp(V),h\exp(W)).$$
The following theorem is then straightforward to check.
\begin{theorem}
The groupoid $\hblup(G\ltimes M)$ is a metrizable locally compact continuous family groupoid in the sense \cite{ContFamilyGroupoids}.
\end{theorem}
The groupoid $\hblup(G\ltimes M)$ being a continuous family groupoid is essentially the content of Theorem \ref{closed2}. The following proposition is straightforward to prove.\begin{prop} One has a groupoid morphism $G\ltimes \blup(G\ltimes M)\to \hblup(G\ltimes M)$ defined by $(g,V)\mapsto g\exp(V)$. This groupoid morphism is surjective, its kernel is $N=\{(g,V):g\in \exp(V)\}$. The induced map $\left(G\ltimes \blup(G\ltimes M)\right)/N\to  \hblup(G\ltimes M)$ is an isomorphism.
\end{prop}
\paragraph{Lie algebroid.}
The groupoid $\hblup(G\ltimes M)$ being a continuous family groupoid has a Lie algebroid. The Lie algebroid can be reconstructed as follows. Let $(V,x)\in \blup(G\ltimes M)$. Then $\hblup(G\ltimes M)_{(V,x)}\simeq G/\exp(V)$. \begin{prop}\label{prop: Lie algebroid groups} The bundle over $\blup(G\ltimes M)$ whose fiber over $(V,x)$ is $$T_{(V,x)}\hblup(G\ltimes M)_{(V,x)}\simeq \frac{\mathfrak{g}}{V}$$ is a vector bundle.
\end{prop}
\begin{proof}
Consider the inclusion $i:\blup(G\ltimes M)\to \GR(\mathfrak{g})\times M$. Let $E$ be the canonical vector bundle on $\GR(\mathfrak{g})\times M$ whose fiber at $(V,x)\in \GR(\mathfrak{g})\times M$ is equal to $V$, $F$ the trivial vector bundle with fiber $\mathfrak{g}$. Then $i^*(F/E)$ is the above vector bundle.
\end{proof}

\section{Blowup of singular foliations}\label{sec:general case}
\subsection{Singular foliations and fibers}
Let $M$ be a smooth manifold, $\cF$ a $C^\infty(M,\R)$-submodule of $\Gamma_c(TM)$. If $X_1,\cdots,X_k\in \cF$, $U\subseteq M$, then we say that $\cF$ is generated by $X_1,\cdots,X_k$ on $U$ if for any $X\in \cF$, there exists $f_1,\cdots,f_k\in C^\infty(U,\R)$ such that $$X=\sum_{i=1}^k f_i X_{i},\quad \text{on}\,\, U$$ 
\begin{dfn}A singular foliation on $M$ is a $C^\infty(M,\R)$-submodule $\cF$ of $\Gamma_c(TM)$.  such that \begin{enumerate}
\item $\cF$ is closed under Lie bracket, i.e., if $X,Y\in \cF$, then $[X,Y]\in \cF$.
\item $\cF$ is locally finitely generated, i.e., $M$ can be covered by open sets $(U_i)_{i\in I}$ such that on each $U_i$, there exists a finite family $X_1,\cdots,X_k\in\cF$ which generates $\cF$ on $U_i$ ($k$ can depend on $i$).
\end{enumerate}
\end{dfn}
\begin{ex}\label{ex:singular foliation group action} Let $G$ be a Lie group acting on $M$, $\natural:\mathfrak{g}\to \Gamma(TM)$ the anchor map. One has a singular foliation $\cF$ on $M$ defined by $$\cF=\{\sum_{i=1}^kf_i\natural(X_i):k\in \N,f_i\in C^\infty_c(M,\R),X_i\in\mathfrak{g}\}.$$ \end{ex}

Let $\cF$ be a singular foliation and $x\in M$. We define two fibers of $\cF$ at $x$ by $$\cF_x:=\frac{\cF}{I_x\cF},\quad F_x:=\{X(x)\in T_xM:X\in \cF\},$$ where $I_x\subseteq C^\infty(M,\R)$ is the ideal of smooth functions vanishing at $x$. If $X\in \cF$, then we denote by $[X]_x\in \cF_x$ the class of $X$. Let $\natural_x$ be the linear map $$\natural_x:\cF_x\to F_x,\quad \natural_x([X]_x)=X(x),\, \forall \,X\in \cF.$$ \begin{rem}\label{rem:Lie algebra hx}The Lie bracket of vector fields endows $\mathfrak{h}_x:=\ker(\natural_x)$ with the structure of a Lie algebra, given by $$[[X]_x,[Y]_x]:=[[X,Y]]_x,\quad X,Y\in \cF.$$\end{rem}
\begin{prop}[\cite{AS1}]\label{prop:AS}\begin{enumerate}
\item Let $x\in M$, $X_1,\cdots,X_k\in \cF$. Then the family $X_1,\cdots ,X_k$ generate $\cF$ on a neighbourhood of $x$ if and only if $[X_1]_x,\cdots ,[X_k]_x$ generate $\cF_x$ as a vector space. Furthermore $k$ is minimal if and only if $[X_1]_x,\cdots,[X_k]_x\in \cF_x$ form a basis of $\cF_x$.
\item The function $$M\to \mathbb{N}\sqcup\{0\},\quad x\to \dim(\cF_x)$$ is upper semi-continuous. The set of its points of continuity will be denoted by $M_{\areg}$.
\item The function $$M\to \mathbb{N}\sqcup\{0\},\quad x\to \dim(F_x)$$ is lower semi-continuous. The set of its points of continuity will be denoted by $M_{\reg}$.
\item One has $M_{\reg}\subseteq M_{\areg}$ and both sets are open and dense.
\item The set $M_{\reg}$ is equal to the set of points such that $\natural_x$ is an isomorphism.
\end{enumerate}
\end{prop}
 In general the inclusion $M_{\reg}\subseteq M_{\mathrm{areg}}$ is strict, for example if $M=\R$, $\cF=\langle x\partial_x\rangle$, then $M_{\reg}=\R\backslash\{0\}$ and $M_{\mathrm{areg}}=\R$.
\begin{theorem}[\cite{Sussmann,Stefan}]Let $x\in M$. Then there exists a unique immersed connected smooth submanifold $l\subseteq M$ that contains $x$ such that $T_yl=F_y$ for all $y\in l$. The manifold $l$ is called the leaf of $x$.
\end{theorem}
\subsection{Blow-up space}
 Let $x\in M$, $k=\dim(\cF_x)$, $X_1,\cdots,X_k\in\cF$ be a family such that $[X_1]_x,\cdots,[X_k]_x\in \cF_x$ form a basis. By Proposition \ref{prop:AS}, it follows that \begin{itemize}
\item There exists a neighbourhood $U$ of $x$ such that $X_1,\cdots,X_k\in\cF$ generate $\cF$ on $U$.
 \item if $y\in U$, then $[X_1]_y,\cdots,[X_k]_y$ generate $\cF_y$.
 \end{itemize} Let $\phi_y$ be the surjective linear map defined by \begin{align}\label{dfn phi}
\phi_y:\cF_x\to \cF_y,\quad \phi_y([X_i]_x)=[X_i]_y,\quad \forall i\in\{1,\cdots,k\}.
\end{align}

\begin{dfn}The set $\cblup(\cF)_x$ (and respectively $\blup(\cF)$) denotes the set of $V\subseteq \cF_x$ such that there exists a sequence $x_n\in M$ ($x_n\in M_{\reg}$) such that \begin{equation}\label{eqn:conv blowup}
x_n\to x,\quad \phi_{x_n}^{-1}(\mathfrak{h}_{x_n})\to V.
\end{equation} 
\end{dfn}
All of the results of this article (except Proposition \ref{prop:gamma foliation}.3 and the results in Section \ref{sec:semicont fields}) work equally well if one literally replaces $\blup(\cF)$ with $\cblup(\cF)$. In our opinion, the space $\blup(\cF)$ is the more natural space to consider.
\begin{prop}\label{prop:gamma foliation}
For $x\in M$, the set $\blup(\cF)_x$ is a non-empty closed subset of $\GR(\cF_x)$ which doesn't depend on the choice of $X_1,\cdots,X_k$. Furthermore if $V\in \blup(\cF)_x$, then \begin{enumerate}
\item One has $V\subseteq \mathfrak{h}_x$.
\item The vector subspace $V$ is a Lie subalgebra of $\mathfrak{h}_x$.
\item One has $x\in M_{\areg}$ if and only if $\blup(\cF)_x=\{0\}$.
\end{enumerate}
\end{prop}\begin{proof}
By compactness of the Grassmannian spaces, $\blup(\cF)_x$ is  non-empty. For independence, let $Y_1,\cdots,Y_k\in \cF$ such that $[Y_1]_x,\cdots ,[Y_k]_x$ form a basis of $\cF_x$. Suppose $U$ is small enough that $\cF$ is generated by $X_1,\cdots,X_k$ on $U$ and is also generated by $Y_1,\cdots,Y_k$ on $U$. It follows that there exists smooth functions $f_{ij}\in C^\infty(U,\R)$ such that $$X_i=\sum_{j=1}^k f_{ij}Y_j,\quad (\text{on }U).$$ Let $\psi_y:\cF_x\to \cF_y,L_y:\cF_x\to \cF_x$ be the linear maps defined by $$\psi_y([Y_i]_x)=[Y_i]_y,\quad L_y([X_i]_x)=\sum_{j=1}^kf_{ij}(y)[Y_j]_x.$$ Clearly $\psi_y\circ L_y=\phi_y$. Since $L_y$ converges to the identity as $y\to x$, the result follows. 

The inclusion $V\subseteq \mathfrak{h}_x$ follows directly from Eqn. \eqref{eqn:conv blowup}. We now prove that $V$ is a Lie subalgebra. Since $\cF$ is closed under Lie brackets, there exists smooth functions $f_{ij}^l\in C^\infty(U,\R)$ such that \begin{equation}\label{eqn:6 isom}
 [X_i,X_j]=\sum_{l=1}^k f_{ij}^lX_l,\quad (\text{on }U).
\end{equation} Let $A,B\in V$ with $$A=\sum_{i=1}^ka_i[X_i]_x,\quad B=\sum_{i=1}^kb_i[X_i]_x.$$ Let $x_n\in M_{\reg}$  as in Eqn. \eqref{eqn:conv blowup}. Hence there exists sequences $a_1^n,\cdots,a_k^n,b_1^n,\cdots,b_k^n$ such that $$a_i^n\xrightarrow{n\to \infty} a_i,\quad b_i^n\xrightarrow{n\to \infty} b_i,\quad \forall i\in \{1,\cdots,k\},\quad \text{and}\quad \sum_{i=1}^ka_i^nX_i(x_n)=\sum_{i=1}^kb_i^nX_i(x_n)=0.$$ Therefore \begin{align}\label{eqn:8 isom}
 \sum_{i,j}a_i^nb_j^n[X_i,X_j](x_n)=\left[\sum_{i=1}^k a_i^nX_i,\sum_{i=1}^kb_i^nX_i\right](x_n)=0.
\end{align} By Equation \eqref{eqn:6 isom} and \eqref{eqn:8 isom}, one gets $$\sum_{i,j,l=1}^ka_i^nb_j^nf_{ij}^l(x_n)X_l(x_n)=0.$$ Hence $\sum_{i,j,l=1}^ka_i^nb_j^nf_{ij}^l(x_n)[X_l]_x\in \phi_{x_n}^{-1}(\mathfrak{h}_{x_n})$ and thus $\sum_{i,j,l=1}^ka_ib_jf_{ij}^l(x)[X_l]_x\in V$. By Eqn. \eqref{eqn:6 isom}, $$\left[[\sum_{i=1}^ka_i X_i]_x,[\sum_{i=1}^kb_i X_i]_{x}\right]=\sum_{i,j,l=1}^ka_ib_jf_{ij}^l(x)[X_l]_x\in V.$$

If $x\in M_{\areg}$, then for $y$ close enough to $x$, $\phi_y:\cF_x\to \cF_y$ is an isomorphism. If $y\in M_{\reg}$, then $\phi_y^{-1}(\mathfrak{h}_y)=\ker(\phi_y)=0$. Hence $\blup(\cF)_x=\{0\}.$ Compactness of Grassmannian spaces and $\blup(\cF)_x=\{0\}$ imply that $x\in M_{\areg}$.
\end{proof}
Proposition \ref{prop:gamma foliation}.3 is false for $\cblup(\cF)$, because $\mathfrak{h}_x\in\cblup(\cF)_x$ for all $x\in M$. It is true however that $x\in M_{\reg}$ if and only if $\cblup(\cF)_x=\{0\}$. 
\paragraph{Topology.} Let $\blup(\cF)=\{(V,x):x\in M,V\in \blup(\cF)_x\}$. We equip $\blup(\cF)$ with the topology such that \begin{itemize}
\item the natural projection $\pi_{\blup(\cF)}:\blup(\cF)\to M$ is continuous \item if $U\subseteq M$ is an open neighbourhood of $x\in M$, and $X_1,\cdots,X_k\in \cF$ as above, then the inclusion \begin{align}\label{eqn:top blowup}
 \pi^{-1}_{\blup(\cF)}(U)\cap \blup(\cF)\to \GR(\cF_x)\times U,\quad (V,y)\to (\phi_y^{-1}(V),y)
\end{align} is an embedding.
\end{itemize} 
It is straightforward to check that the topology on $\blup(\cF)$ is well defined (different local generators give compatible maps) and that it is a locally compact and Hausdorff.

The topology on $\blup(\cF)$ can also be defined intrinsically as follows. In \cite{AS2}, the authors equip the space $\cF^*:=\sqcup_{x\in M}\cF_x^*$ with the weakest topology such that\begin{itemize}
 \item the projection $\pi:\cF^*\to M$ is continuous
 \item if $X\in \cF$, then the natural map $\cF^*\to \R$ defined by $\xi\in \cF_x^*\mapsto \langle \xi,[X]_x\rangle$ is continuous.
\end{itemize} The topology on $\blup(\cF)$ is the weakest topology such that a net $(V_n,x_n)\to (V,x)$ if and only if \begin{itemize}
\item $x_n\to x$ and $\dim(V_n^\perp)\to \dim(V^\perp)$.\footnote{The function $\blup(\cF)\to \Z,\quad (V,x)\mapsto \dim(V)$ isn't continuous, only $x\mapsto \dim(V^\perp)=\dim(\cF_x)-\dim(V)$ is.}
\item for every $\xi_n\in \cF_{x_n}^*$, if  $\xi_n\in \cF^*$ converges to $\xi\in \cF_x^*\subseteq \cF^*$ and $\xi_n\in V_n^\perp$ for all $n$, then $\xi\in V^\perp$.
\end{itemize} 
\begin{exs}\label{exs 1}\begin{enumerate}
\item Let $N\subseteq M$ be a smooth submanifold, $\cF$ the module of compactly supported vector fields on $M$ which vanish on $N$. Then one sees that \begin{align*}
\cF_x=\begin{cases}T_xM&\text{if}\, x\notin N\\\mathrm{Hom}(\frac{T_xM}{T_xN},T_xM)&\text{if}\, x\in N\end{cases},\quad F_x=\begin{cases}T_xM& \text{if}\, x\notin N\\0&\text{if}\, x\in N\end{cases}.
\end{align*}
Hence $M_{\reg}=M_{\areg}=N^c$. A direct computation shows that for $x\in N$, 
$$\blup(\cF)_{x}=\{\ker(v):v\in \frac{T_xM}{T_xN},v\neq 0\},\quad \ker(v)=\{L\in \mathrm{Hom}(\frac{T_xM}{T_xN},T_xM):L(v)=0\}.$$ Hence $\blup(\cF)_x$ can be identified with the projective space $\mathbb{P}(T_xM/T_xN)$. Furthermore the space $\blup(\cF)$ is then homeomorphic to the classical blow-up of $M$ along $N$. One also checks that $\cblup(\cF)=\blup(\cF)\sqcup N$ with the disjoint union topology.
\item If $G$ is a connected Lie group acting on $M$ with anchor map $\natural$. One has a surjective map $\tilde{\natural}_x:\mathfrak{g}\to \cF_x$ given by $\tilde{\natural}_x(X)=[\natural(X)]_x$. It is straightforward to check that if $V\in \blup(G\ltimes M)_x$, then $\ker(\tilde{\natural}_x)\subseteq V$ and $\blup(\cF)_x=\{\tilde{\natural}_x(V):V\in \blup(G\ltimes M)\}$.
\end{enumerate}
\end{exs}
\begin{rem}\label{rem:open}The map $\pi_{\cblup(\cF)}:\cblup(\cF)\to M$ isn't necessarily open as Example \ref{exs 1}.1 shows. I don't have an example where $\pi_{\blup(\cF)}:\blup(\cF)\to M$ isn't open.
\end{rem}
\paragraph{The bundle $T\cF$} Since the space $\blup(\cF)$ embeds locally inside a Grassmannian space. The bundle over $\blup(\cF)$ whose fiber over $(V,x)$ is $\cF_x/V$ is naturally a vector bundle. It satisfies the properties stated in Theorem \ref{intro thm}.2 (with $U=M_{\reg}$.)
\paragraph{Functoriality of the blowup space.}Let $\phi:N\to M$ be a smooth submersion, $\cF$ a singular foliation on $M$. Then $\phi^{-1}(\cF)$ denotes the singular foliation on $N$ of all $X\in \Gamma_c(TN)$ such that there exist  $k\in \N$ and $f_1,\cdots ,f_k\in C^\infty(N),X_1,\cdots ,X_k\in \cF$  such that\begin{equation}\label{eqn:proof frak d def 3}
 d\phi(X)=\sum_i f_i X_i\circ \phi.
\end{equation} By \cite[Proposition 1.10]{AS1}, $\phi^{-1}(\cF)$ is a singular foliation. Let $x\in N$. Then there is a natural map $$\mathfrak{d}_x\phi:\phi^{-1}(\cF)_x\to \cF_{\phi(x)}$$ defined as follows. Let $X\in \phi^{-1}(\cF)_x$, $f_i\in C^\infty(N)$, $X_i\in \cF$ as in Eqn. \eqref{eqn:proof frak d def 3}. We define $$\mathfrak{d}_x\phi([X]_x):=\sum_i f_i(x)[X_i]_{\phi(x)}\in \cF_{\phi(x)}.$$ \begin{prop}\label{prop:Funct pullback} The map $\mathfrak{d}_x\phi$ is a well defined surjective linear map and 
\begin{equation}\label{eqn:proof math frak d}
 \natural_{\phi(x)}\circ \mathfrak{d}_x\phi=d_{x}\phi\circ \natural_{x}.
\end{equation} 
\end{prop}
\begin{proof}We will show that $\mathfrak{d}_x\phi([X]_x)$ doesn't depend on the choice of $f_i,X_i$. Suppose that $g_i\in C^\infty(N),Y_i\in \cF$ such that \begin{equation}\label{eqn:proof frak d def}
 d\phi(X)=\sum f_i X_i\circ \phi=\sum g_i Y_i\circ \phi.
\end{equation} Let $h:\dom(h)\subseteq M\to N$ be any local section of $\phi$ around $\phi(x)$. We precompose Eqn. \eqref{eqn:proof frak d def} with $h$ to obtain \begin{equation}\label{eqn:proof frak d def 2}
\sum (f_i\circ h) X_i=\sum_i (g_i\circ h) Y_i\end{equation}
Both sides of Equation \eqref{eqn:proof frak d def 2} are elements of $\cF$ (defined locally around $\phi(x)$) and hence we can take their class in $\cF_{\phi(x)}$. We get $$\sum f_i(x) [X_i]_{\phi(x)}=\sum g_i(x) [Y_i]_{\phi(x)}.$$ This shows that $\mathfrak{d}_x\phi([X]_x)$ only depends on $X$. To show only dependence on $[X]_x$, let $X'\in \phi^{-1}(\cF)$ such that $[X']_{x}=[X]_x$. Then $X'-X\in I_x\phi^{-1}(\cF)$. Hence there exists $l_i\in C^\infty(N),Y_i\in\phi^{-1}(\cF)$ such that $X'-X=\sum_i l_i Y_i$ and $l_i(x)=0$. By writing each $d\phi(Y_i)$ as in Eqn. \eqref{eqn:proof frak d def 3}, it follows that one can write $X'-X$ as in Eqn. \eqref{eqn:proof frak d def 3} with each $f_i(x)=0$. It follows that $\mathfrak{d}_x\phi([X]_x)$ is well defined. 

It is clear that $\mathfrak{d}_x\phi([X]_x)$ is linear surjective. For Eqn. \eqref{eqn:proof math frak d}, one has $$\natural_{\phi(x)}(\mathfrak{d}_x\phi([X]_x))=\natural_{\phi(x)}\left(\sum_i f_i(x)[X_i]_{\phi(x)}\right)=\sum_i f_i(x)X_i(\phi(x))=d_x\phi(X(x)).$$This finishes the proof.%
\end{proof}
We remark that $\Gamma(\ker(d\phi))\subseteq \phi^{-1}(\cF)$. By taking the localization of each module at $x$ and using the fact that $\ker(d\phi)$ is a vector bundle, one obtains a linear map $$\ker(d\phi)_x\to \phi^{-1}(\cF)_{x}.$$\begin{prop}\label{prop:short exact seq fibers pullback}One has a short exact sequence $$0\to \ker(d\phi)_x\to \phi^{-1}(\cF)_{x}\xrightarrow{\mathfrak{d}_{x}\phi}\cF_{\phi(x)}\to 0.$$
\end{prop}
\begin{proof}
The statement is local, therefore we can suppose that $N=M\times \R^k$, $\phi$ is the projection. If $X\in \cF$, then let $\bar{X}\in \Gamma(TN)$ be the constant extension defined by $\bar{X}(x,t)=X(x)$. It is clear that $$\phi^{-1}(\cF)=\{\sum_{i=1}^k f_i\partial_{t_i}+\sum_{j=1}^s g_j \bar{X}_j:s\in \N, f_i,g_i\in C^\infty_c(N),X_j\in \cF\}.$$ The proposition is then straightforward to check.
\end{proof}
\begin{theorem}[Functoriality of the blowup]\label{thm:funct blowup}Let $x\in N$. The maps  \begin{align*}
\blup(\phi^{-1}(\cF))_x\xrightarrow{\blup(\phi)_x} \blup(\cF)_{\phi(x)},\quad V\mapsto \mathfrak{d}_x\phi(V).
\end{align*}and \begin{align*}
\blup(\phi^{-1}(\cF))\xrightarrow{\blup(\phi)} \blup(\cF)\times_{\pi_{\blup(\cF)},\phi}N,\quad (V,x)\mapsto (\mathfrak{d}_x\phi(V),\phi(x)).
\end{align*} are homeomorphisms.\end{theorem}
\begin{proof}
Let $k=\dim(\cF)$, $k'=\dim(\phi^{-1}(\cF))$, $X_1,\cdots,X_k\in \cF$ such that $[X_1]_x,\cdots,[X_k]_x$ form a basis of $\cF_x$. Let $\tilde{X}_1,\cdots,\tilde{X}_k\in \phi^{-1}(\cF)$ such that $d\phi(\tilde{X}_i))=X_i\circ \phi$, $\tilde{X}_{k+1},\cdots,\tilde{X}_{k'}\in \Gamma(\ker(d\phi))$ such that $\tilde{X}_{k+1}(x),\cdots,\tilde{X}_{k'}(x)$ form a basis of $\ker(d\phi)_x$. If one uses family $X_1,\cdots,X_k$ and $\tilde{X}_1,\cdots,\tilde{X}_{k'}$ to calculate $\blup(\cF)$ and $\blup(\phi^{-1}(\cF))$ the theorem follows immediately.
\end{proof}
\subsection{Holonomy groupoid and action on the blowup space}
In this section, we will recall the construction of the holonomy groupoid $\cH(\cF)$ in \cite{AS1}. We then show that it acts on $\blup(\cF)$.
\paragraph{Bi-submersions.} A bi-submersion for $\cF$ is a triple $(U,r_U,s_U)$ where $U$ is a smooth manifold, $s_U,r_U:U\to M$ are submersions such that \begin{equation}\label{eqn:bi-sub}
 s_U^{-1}(\cF)=r_U^{-1}(\cF)=\Gamma(\ker(ds_U))+\Gamma(\ker(dr_U)).
\end{equation}
If $x,y\in M$, then we denote $s_U^{-1}(x)$ by $U_x$, $r_U^{-1}(y)$ by $U^y$, $s_U^{-1}(x)\cap r_U^{-1}(y)$ by $U^y_x$. We will often simply write $U$ instead of $(U,r_U,s_U)$ for a bi-submersion.
\begin{exs}\label{exs:bisubmersions}\begin{enumerate}
\item If $\cF$ is as in Example \ref{ex:singular foliation group action}, then by \cite[Proposition 2.2]{AS1}, $(G\times M,r,s)$ is a bi-submersion where $r(g,x)=g\cdot x$, $s(g,x)=x$.
\item Let $x\in M$, $\cF$ an arbitrary singular foliation, $X_1,\cdots,X_k\in \cF$ such that $[X_1]_x,\cdots,[X_k]_x$ form a basis of $\cF_x$. We define \begin{align*}
r,s:M\times \R^k\to M,\quad s(y,t_1,\cdots,t_k)=y,\quad  r(y,t_1,\cdots,t_k)=\exp_{\sum_{i=1}^kt_iX_i}(y)
\end{align*}By \cite[Proposition 2.10]{AS1}, for some open neighbourhood $W$ of $(x,0)\in M\times \R^k$, the triple $(W,r_{|W},s_{|W})$ is a bi-submersion.
\end{enumerate} 
\end{exs}
If $(U,r_U,s_U)$ is a bisubmersion. Then its inverse $U^{-1}$ is the bi-submersion $(U,s_U,r_U)$. If $U,U'$ are bi-submersions, then their product $U\circ U'$ is the bi-submersion $(U\times_{s_U,r_{U'}}{U'},r_{U\circ {U'}},s_{U\circ {U'}})$, where $r_{U\circ {U'}}(u,u')=r_U(u)$ and $s_{U\circ {U'}}(u,u')=s_{U'}(u')$.
\paragraph{Bi-sections and path holonomy.} A bi-section of a bi-submersion $U$ at $u\in U$ is an embedded submanifold $S\subseteq U$ such that $u\in S$ and $$s_{U |S}:S\to s_U(S),\quad r_{U|S}:S\to r_U(S)$$ are diffeomorphisms onto open subsets of $M$. By \cite[Proposition 2.7]{AS1}, bi-sections always exists. It follows from the definition of a bi-submersion that $r_{U|S}\circ s_{U|S}^{-1}$ is an automorphism of $\cF$ (on its domain of definition), called the automorphism associated to $S$. \textbf{We will suppose that all bi-submersions are path holonomy}. This means that for any $u\in U$, there exists a bi-section $S$ such that $r_{U|S}\circ s_{U|S}^{-1}=\phi_{|s_U(S)}$ for some $\phi\in \exp(\cF)$, the group of diffeomorphisms of $M$ generated by $\exp_X$ for $X\in \cF$, see \cite[Proposition 1.6]{AS1}. We remark that this isn't a restrictive condition on $\cF$ because path holonomy bi-submersions always exist. The bi-submersions in Examples \ref{exs:bisubmersions} are path holonomy.

We say that a bi-submersion carries the identity at $u$ if there exists a bi-section at $u$ such that the associated automorphism is equal to the identity. We call the bi-submersion $U$ minimal at $x$ if $\dim(U)=\dim(M)+\dim(\cF_{x})$. The bi-submersion in Example \ref{exs:bisubmersions}.2 carries the identity and is minimal at $(x,0)$.
\paragraph{Morphisms of bi-submersions.} If $(U,r_U,s_U),(U',r_{U'},s_{U'})$ are bi-submersions, then a morphism of bi-submersions $f:(U,r_{U},s_{U})\to (U',r_{U'},s_{U'})$ is a smooth map $f:U\to U'$ such that $$s_{U'}\circ f=s_U,\quad r_{U'}\circ f=r_U.$$
A local morphism at $u\in U$ is a morphism which is defined in a neighbourhood of $u$.
\paragraph{Holonomy Groupoid.}Consider the disjoint union of all bi-submersions $\sqcup_{U} U=\{(U,u)\}$. Here $U$ is a bi-submersion\footnote{To be a set, one has to restrict to bi-submersions which are subsets of $\R^\infty$.}, $u\in U$. We equip this set with a relation $(U,u)\sim (U',u')$ if there exists a local morphism of bi-submersions $f:U\to U'$ such that $f(u)=u'$. By \cite[Corollary 2.11]{AS1}, this is an equivalence relation. The quotient of the set $\sqcup_{U} U$ by this relation is the holonomy groupoid and is denoted by $\cH(\cF)$. One equips $\cH(\cF)$ with the quotient of the disjoint union topology on $\sqcup_{U} U$.

If $U$ is a bi-submersion, we denote by $q_U:U\to \cH(\cF)$ the map which sends $u$ to the class of $(U,u)$. The groupoid structure on $\cH(\cF)$ is given by \begin{align*}
s_{\cH(\cF)}(q_U(u))=s_U(u),\quad r_{\cH(\cF)}(q_U(u))=r_U(u)\\
q_U(u)q_{U'}(u')=q_{U\circ U'}(u,u'),\quad q_U(u)^{-1}=q_{U^{-1}}(u).
\end{align*}
The identity at $x\in M$ is given by $q_U(u)$ where $U$ is any bi-submersion which carries the identity at $u$ with $s_U(u)=x$. 

\begin{rems}\label{rem:open map}\begin{enumerate}
\item Even though the topology on $\cH(\cF)$ is quite pathological, it is shown in \cite{Debord2013} that for every $x\in M$, the map $q_{U|U_x}:U_x\to \cH_x$ is a local homeomorphism, where $U$ is a bi-submersion which is minimal at $x$. Hence $\cH(\cF)_x$ naturally carries a unique smooth manifold structure of dimension $\dim(\cF_x)$ such that the maps $q_{U|U_x}:U_{x}\to \cH(\cF)_x$ are smooth submersions when $U$ is a bi-submersion. Furthermore because we only consider path-holonomy bi-submersions, $\cH(\cF)_x$ is connected.\begin{theorem}[\cite{AndZambon1}]\label{thm:AZ}Let $x\in M$, $l\subseteq M$ the immersed leaf which passes though $x$. Then \begin{enumerate}
\item the connected component of $\cH(\cF)_x^x$ at $x$ is a simply connected Lie group whose Lie algebra is $\mathfrak{h}_x$.
\item the map $r:\cH(\cF)_x\to l$ is a $\cH(\cF)_x^x$ principal bundle.
\end{enumerate}
\end{theorem}
\item By \cite[Corollary 3.12]{AS1}, the map $q_{U}$ is an open map for any bi-submersion $U$.
\end{enumerate}
\end{rems}
\paragraph{Action of $\cH(\cF)$ on $\blup(\cF)$.}Let $G$ be a topological groupoid. Recall that an action of $G$ on a topological space $X$ consists of continuous maps $\pi: X \to G_0$ and $m: G\times_{s_G,\pi}X\to X$, written $(\gamma, x)\to  \gamma x$, such that $\pi(x)x = x$, $\pi(\gamma x) = r(\gamma)$, and $(\gamma'\gamma)x = \gamma'(\gamma x)$ for all
$x\in X, \gamma\in  G_{\pi(x)}, \gamma' \in G_{r(\gamma)}$.

Let $(U,r_U,s_U)$ be a bi-submersion, $u\in U$. Since $s^{-1}_U(\cF)=r_U^{-1}(\cF)$, by Theorem \ref{thm:funct blowup}, the map $$\blup(\cF)_{s_U(u)}\xrightarrow{\blup(s_U)_u^{-1}}\blup(s_U^{-1}(\cF))_u=\blup(r_U^{-1}(\cF))_u\xrightarrow{\blup(r_U)_u}\blup(\cF)_{r_U(u)}$$ is a homeomorphism. 
\begin{theorem}\label{thm:action} The groupoid $\cH(\cF)$ acts on $\blup(\cF)$ by the formula $$q_U(u)\cdot (V,s_U(u))=(\blup(r_U)_{u}\circ\blup(s_U)_u^{-1}(V),r_U(u)),$$ where $U$ is a bi-submersion, $u\in U$, $V\in \blup(\cF)_{s_U(u)}$.
\end{theorem}
If $\gamma\in \cH(\cF)$, $V\in \blup(\cF)_{s_{\cH(\cF)}(\gamma)}$, then we will write $\gamma\cdot V\in \blup(\cF)_{r_{\cH(\cF)}(\gamma)} $ for the action.
\begin{proof}
Let $U,U'$ be bi-submersions, $u\in U$, $f:U\to U'$ a local morphism, $V\in \blup(\cF)_{s_U(u)}$. We claim that $$\blup(r_U)_{u}\circ\blup(s_U)_u^{-1}(V)=\blup(r_{U'})_{f(u)}\circ\blup(s_{U'})_{f(u)}^{-1}(V).$$ This is obvious if $f$ is a submersion at $u$. By \cite[Corollary 3.12]{AS1}, there exists a bi-submersion $U^{\prime \prime}$ and a local morphism $g:U'\to U^{\prime \prime}$ at $f(u)$ such that $g$ and $g\circ f$ are submersions at $f(u)$ and $u$ respectively. The claim follows. This shows that the action is well defined. We leave checking the algebraic identities to the reader. Continuity of the action follows from Theorem \ref{thm:funct blowup}.
\end{proof}
It follows that $\cH(\cF)\ltimes \blup(\cF)$ is a well defined groupoid, see \cite[Section 2.3]{DebordSkandLiegroup}. We will use the notation $$\cH(\cF)\ltimes \blup(\cF)=\cH(\cF)\times_{{s_{\cH(\cF)}},\pi_{\blup(\cF)}}\blup(\cF)=\{(\gamma,V):\gamma\in\cH(\cF),V\in \blup(\cF)_{s_{\cH(\cF)}(\gamma)}\}.$$
\subsection{Blowup groupoid}\label{sec:smoothnes foliations}
In this section, we define the holonomy blowup groupoid.
\paragraph{Integration of $\blup(\cF)$.}Let $(U,r_U,s_U)$ be a bi-submersion, $u\in U$. Since $\ker(ds_U)_u\subseteq s^{-1}_U(\cF)_u$, $\ker(dr_U)_u\subseteq r_U^{-1}(\cF)_u$ and $s_U^{-1}(\cF)=r_U^{-1}(\cF)$, it follows that one has two natural inclusions given by \begin{align*}
&\ker(dr)_u\cap \ker(ds)_u\hookrightarrow  \ker(dr)_u\hookrightarrow r^{-1}_U(\cF)_u\\&\ker(dr)_u\cap \ker(ds)_u\hookrightarrow  \ker(ds)_u\hookrightarrow s^{-1}_U(\cF)_u.
\end{align*}
It is important in what follows to remark that the two inclusions \textit{aren't equal} in general. We accordingly define two maps $$\mathfrak{d}'_us_U:\ker(dr)_u\cap \ker(ds)_u\to \cF_{s_U(u)},\quad \mathfrak{d}'_ur_U:\ker(dr)_u\cap \ker(ds)_u\to \cF_{r_U(u)}$$ by the composition respectively\begin{align*}
&\ker(dr)_u\cap \ker(ds)_u\hookrightarrow \ker(dr)_u\hookrightarrow r^{-1}_U(\cF)_u=s^{-1}_U(\cF)_u\xrightarrow{\mathfrak{d}s_U} \cF_{s_U(u)}\\&\ker(dr)_u\cap \ker(ds)_u\hookrightarrow \ker(ds)_u\hookrightarrow s^{-1}_U(\cF)_u=r^{-1}_U(\cF)_u\xrightarrow{\mathfrak{d}r_U} \cF_{r_U(u)}.
\end{align*}
\begin{prop}\label{prop:image hx} The image of $\mathfrak{d}'_us_U$ is equal to $\mathfrak{h}_{s_U(u)}$. The kernel is equal to the set of $v\in \ker(dr)_u\cap \ker(ds)_u$ such that there exists $X\in \Gamma(\ker(ds_U))\cap \Gamma(\ker(dr_U))$ such that $X(u)=v$. Similar statement for $\mathfrak{d}'_ur_U$ 
\end{prop}
\begin{proof}
The map $\mathfrak{d}_us_{U|\ker(dr_U)}:\ker(dr_U)_u\to \cF_{s_U(u)}$ is surjective because $s_U^{-1}(\cF)=\ker(ds_U)+\ker(dr_U)$. By Eqn. \eqref{eqn:proof math frak d}, the image inverse of $\mathfrak{h}_{s_U(u)}$ is equal to $\ker(dr)_u\cap \ker(ds)_u$. To compute the kernel, let $v\in \ker(dr)_u\cap \ker(ds)_u$ which is in kernel of $\mathfrak{d}'_us_U$. By Proposition \ref{prop:short exact seq fibers pullback}, this means that one can find $X\in \Gamma(\ker(ds_U))$ and $Y\in \Gamma(\ker(dr_U))$ such that $X(u)=Y(u)=v$ and $[X]_u=[Y]_u\in s^{-1}_U(\cF)_{s_U(u)}$. Hence $X-Y=\sum f_i Z_i$ with $Z_i\in s^{-1}_U(\cF) $ and $f_i\in C^\infty(U,\R)$ vanish at $u$. Since $s_U^{-1}(\cF)=\ker(ds_U)+\ker(dr_U)$, it follows that each $Z_i=A_i+B_i$ with $A_i\in \ker(ds_U),B_i\in \ker(dr_U)$. It follows that $X-\sum f_iA_i=Y+\sum f_i B_i\in \Gamma(\ker(ds_U))\cap \Gamma(\ker(dr_U))$ and it is equal to $v$ at $u$.
\end{proof}
We remark that a straightforward computation shows that $u\cdot V=\mathfrak{d}'_ur_U(\mathfrak{d}'_us_U^{-1}(V))$.
 \begin{prop}\label{prop:foliation} Let $(V,x)\in \blup(\cF)$. Then
 \begin{enumerate}
 \item The bundle $(\mathfrak{d}'_us_U^{-1}(V))_{u\in U_x}$ is a vector subbundle of $TU_x$ which will be denoted by $\mathfrak{d}'s_U^{-1}(V)$.
 \item The bundle $\mathfrak{d}'s_U^{-1}(V)$ is a regular foliation on $U_x$.
\end{enumerate} 
 \end{prop}
 \begin{proof}The maps $$\mathfrak{d}_us_{U|\ker(dr_U)}:\ker(dr_U)_u\to \cF_x$$ vary smoothly as $u$ varies in $U_x$. To see this, one chooses any local basis of $\Gamma(\ker(dr_U))$ and then write each base element as in Equation \eqref{eqn:proof math frak d}. Hence $\mathfrak{d}'s_U^{-1}(V)$ is a subbundle of $\ker(dr_U)_{|U_x}$. Since $$\mathfrak{d}'_us_U^{-1}(V)\subseteq \ker(dr)_u\cap \ker(ds)_u\subseteq \ker(ds_U)_u=T_uU_x,$$ it follows that $\mathfrak{d}'s_U^{-1}(V)$ is a vector subbundle of $TU_x$ as well.

 Let $L=\Gamma(\ker(ds)_{|U_x})\cap \Gamma(\ker(dr_U)_{|U_x})$, where the intersection is taken as subsets of $\Gamma(TU_{|U_x})$. We claim that the Lie bracket endows $L$ with the structure of a Lie algebra. To see this let $X,Y\in L$. Then \begin{itemize}
 \item from one side $X,Y\in\Gamma(\ker(ds)_{|U_x})$ are both vector fields on $U_x$, hence $[X,Y]$ is well defined and $[X,Y]\in \Gamma(\ker(ds)_{U_x})$.
 \item from the other side, $X,Y\in \Gamma(\ker(dr_U)_{|U_x})$ admit an extension $\tilde{X},\tilde{Y}\in\Gamma(\ker(dr_U))$. Furthermore $[\tilde{X},\tilde{Y}]\in \Gamma(\ker(dr_U))$ because $\ker(dr_U)$ is a regular foliation. Hence the restriction of $[\tilde{X},\tilde{Y}]$ to $U_x$ still belongs to $\Gamma(\ker(dr_U)_{U_x})$. Finally the Lie bracket is independent of the choice of the extensions on $U_x$ and is equal to $[X,Y]$.
 \end{itemize}
 Let $u\in U_x$. We define a map $\phi_u:L\to \mathfrak{h}_x$ by $$L\xrightarrow{\mathrm{ev_u}} \ker(ds_U)_u\cap \ker(dr_U)_u\xrightarrow{\mathfrak{d}'_us_U}\mathfrak{h}_x,$$ where $\ev_u$ is evaluation at $u$. By Remark \ref{rem:Lie algebra hx}, $\mathfrak{h}_x$ is also a Lie algebra.  \begin{lem}\label{lem:temp 1234} Let $X,Y\in L$. Then $u\mapsto\phi_u(X)$ is a smooth function $U_x\mapsto\mathfrak{h}_x$ which we denote by $\phi(X)$. One has $$\phi_u([X,Y])=[\phi_u(X),\phi_u(Y)]-Y\cdot \phi(X)(u)+X\cdot \phi(Y)(u).$$ Here $Y\cdot \phi(X)$ means the differential of the function $\phi(X)$ in the direction of $Y$, which is well defined because $Y$ is a vector field on $U_x$.
 \end{lem}
 \begin{proof}
  Let $X,Y\in L$, $\tilde{X},\tilde{Y}\in \Gamma(\ker(dr_U))$ extensions of $X$ and $Y$. Then $\tilde{X},\tilde{Y}\in r^{-1}_U(\cF)=s^{-1}_U(\cF)$. Hence there exists $f_i,g_i\in C^\infty(U)$, $X_i,Y_i\in \cF$ such that $$ds_U(\tilde{X})=\sum f_i X_i\circ s_U,\quad ds_U(\tilde{Y})=\sum g_i Y_i\circ s_U.$$ 
Hence  $$\phi_u(X)=\sum f_i(u)[X_i]_x,\quad \phi_u(X)=\sum g_i(u)[Y_i]_x.$$ This implies smoothness of $\phi(X)$ and $\phi(Y)$. Since
\begin{align*}
   ds_U([\tilde{X},\tilde{Y}])=\sum_{i,j}f_ig_j [X_i,Y_j]\circ s_U-\sum_i (\tilde{Y}\cdot f_i)X_i\circ s_U+\sum_j (\tilde{X}\cdot g_j)Y_j\circ s_U. 
\end{align*}
It follows that \begin{align*}
 \phi_u([X,Y])&=\sum_{i,j}f_i(u)g_j(u) [[X_i,Y_j]]_x-\sum_i \tilde{Y}\cdot f_i(u)[X_i]_x+\sum_j \tilde{X}\cdot g_j(u)[Y_j]_x\\
 &=\left[\phi_u(X),\phi_u(Y)\right]-Y\cdot \phi(X)(u)+X \cdot\phi(Y)(u).\qedhere
\end{align*}
\end{proof} 
 Since $\Gamma(\mathfrak{d}'s_U^{-1}(V))=\cap_{u\in U_x}\phi_u^{-1}(V)$. Proposition \ref{prop:gamma foliation} and Lemma \ref{lem:temp 1234} imply that $\mathfrak{d}'s_U^{-1}(V)$ is a regular foliation.\qedhere
 \end{proof}
\begin{ex}Consider the bi-submersion in Example \ref{exs:bisubmersions}.1. If $x\in M$, then $U_x=G\times \{x\}$. Let $V\in \blup(\cF)_x$. A straightforward computation shows that the leaf of the foliation $\mathfrak{d}'s_U^{-1}(V)$ which passes through $(g,x)$ is $g\exp(V)\times\{x\}$. 
\end{ex}
 Let $u\in U_x$. Then we denote by $\cL_s(U,u,V)$ the leaf of $\mathfrak{d}'s_U^{-1}(V)$ passing through $u$. We remark that since $\mathfrak{d}'s_U^{-1}(V)\subseteq \ker(dr_U)\cap \ker(ds_U)$, it follows that \begin{equation}\label{eqn:double leaf up}
 \cL_s(U,u,V)\subseteq U_{s_U(u)}^{r_U(u)}.
\end{equation}
Similar to $U_x$, we foliate $U^x$ by the regular foliation $\mathfrak{d}'r_U^{-1}(V)$, whose leaves are denoted by $\cL_r(U,u,V)$.
\begin{theorem}[Leaf decomposition of the holonomy groupoid]\label{thm:leaf decomp}
\begin{enumerate}
\item Let $(V,x)\in \blup(\cF)$, $\gamma\in \cH(\cF)_x$. We define $$\cL_s(\gamma,V)=\cup_{(U,u)}q_U(\cL_s(U,u,V)),$$ where the union is over $U,u$ with $q_U(u)=\gamma$. This leaf decomposition of $\cH(\cF)_x$ comes from a regular foliation. Similarly, we define $\cL_r(\gamma^{\prime},V)$ if $\gamma^{\prime}\in \cH^x$.
\item The leaf decomposition of the $\cH(\cF)_x$ is compatible with the groupoid structure \begin{align}\label{eqn:thm leaf decomp 1}
\gamma'\cL_s(\gamma,V)=\cL_s(\gamma'\gamma,V),\, \cL_s(\gamma,V)\gamma^{\prime\prime}=\cL_s(\gamma\gamma'',\gamma''^{-1}\cdot V), \, \cL_s(\gamma,V)^{-1}=\cL_s(\gamma^{-1},\gamma\cdot V),
\end{align}
where $\gamma'\in \cH_{s_{\cH(\cF)}(\gamma)}$, $\gamma''\in \cH^{r_{\cH(\cF)}(\gamma)}$, and we use the action of $\cH(\cF)$ on $\blup(\cF)$ given in Theorem \ref{thm:action}. The formulas are well defined by Eqn. \ref{eqn:double leaf up}.
\item It is compatible with the action on $\blup(\cF)$ \begin{align}\label{eqn:thm leaf decomp 2}
 \gamma'\in \cL_s(\gamma,V)\Rightarrow \gamma'\cdot V=\gamma\cdot V.
\end{align}
\item The foliation using $\cL_r$ is the related to $\cL_s$ by the identity $\cL_s(\gamma,V)=\cL_r(\gamma,\gamma\cdot V)$.
\end{enumerate} 
\end{theorem}
\begin{cor}The space $N\subseteq \cH(\cF)\ltimes \blup(\cF)$ given by $$N=\{(\gamma,V)):\gamma\in \cL_s(x,V)\}$$ is a subgroupoid. We define the holonomy blowup groupoid $$\hblup(\cF):=\big(\cH(\cF)\ltimes \blup(\cF)\big)/N.$$ We denote the quotient map by $\cQ:\cH(\cF)\ltimes \blup(\cF)\to \hblup(\cF)$.
\end{cor}
We also define $\chblup(\cF):=\big(\cH(\cF)\ltimes \cblup(\cF)\big)/N$, where $N\subseteq \cH(\cF)\ltimes \cblup(\cF)$ given by $N=\{(\gamma,(V,x)):\gamma\in \cL_s(x,V)\}$.
\begin{proof}[Proof of Theorem \ref{thm:leaf decomp}]
\begin{lem}\label{prop:f comm diag}Let $f$ a local morphism of bi-submersion at $u\in U$. The diagram$$\begin{tikzcd}\ker(ds_U)_u\cap\ker(dr_U)_u\arrow[r,"d_uf"]\arrow[dr,"\mathfrak{d}'_us_{U}"']&\ker(ds_{U'})_{f(u)}\cap \ker(dr_{U'})_{f(u)}\arrow[d,"\mathfrak{d}'_{f(u)}s_{U'}"]\\&\cF_{s_U(u)}
\end{tikzcd}$$ commutes. Similarly with $\mathfrak{d}'_ur_{U}$ and $\mathfrak{d}'_{f(u)}r_{U'}$.
\end{lem} 
\begin{proof}
If $f$ is a submersion, then this is obvious. For a general $f$, one uses \cite[Corollary 3.12]{AS1} like in the proof of Theorem \ref{thm:action}. \end{proof}
Let $U$ be a minimal bi-submersion at $x$. Hence $q_{U|U_x}:U_x\to \cH(\cF)_{x}$ is a local diffeomorphism. One can then transport the regular foliation $\mathfrak{d}'s_U^{-1}(V)$ to a regular foliation on $\cH(\cF)_x$. Lemma \ref{prop:f comm diag} implies the regular foliations obtained glue together for different $U$. Furthermore its leaves are $\cL_s(\gamma,V)$ for $\gamma\in \cH(\cF)_x$.

\begin{lem}\label{lem:continuity of leaves}Let $U$ be a bi-submersion, $\GR(\ker(ds_U))$ the fiber bundle over $U$ whose fiber over $u$ is $\GR(\ker(ds_U)_u)$. Then the following map is continuous \begin{align*}
U\times_{s_U,\pi_{\blup(\cF)}}\blup(\cF)\to \GR(\ker(ds_U)),\quad (u,(V,s_U(u)))\to \mathfrak{d}'s_U^{-1}(V)_u
\end{align*}
\end{lem}
The proof of Lemma \ref{lem:continuity of leaves} is very similar to the proof of Proposition \ref{prop:foliation}.1 and will be omitted. 

Since any $(V,x)\in\blup(\cF)$ is a limit of $(0,x_n)\in \blup(\cF)$ with $x_n\in M_{\reg}$, by Lemma \ref{lem:continuity of leaves}, it is enough to prove parts 2,4 for $V=0$. From Theorem \ref{thm:AZ}, it follows that $\cL_s(\gamma,V)=\{\gamma\}$ in $\cH_{s_{\cH(\cF)}(\gamma)}^{r_{\cH(\cF)}(\gamma)}$. Same for $\cL_r$. For part $3$, one also reduces to $V=0$. It then follows for dimension reasons using Proposition \ref{prop:gamma foliation}.1, that $\gamma'\cdot V=\gamma\cdot V=0$. This finishes the proof of Theorem \ref{thm:leaf decomp}. 
\end{proof}
\begin{rem}\label{rem:tangent space of hx}By Lemma \ref{prop:f comm diag}, we can identify the tangent space of $\cH(\cF)_x$  at $\gamma$ with $\cF_{r_{\cH(\cF)}(\gamma)}$, in such a way that if $U$ is a bi-submersion and $u\in U$, then the maps  \begin{align*}
&\mathfrak{d}_ur_{|\ker(ds)}:\ker(ds)_u\to \cF_{r(u)},\quad d_uq_{U|U_x}:\ker(ds)_u\to T_{q_U(u)}\cH(\cF_x)=\cF_{r(u)}
\end{align*} are the same.
\end{rem}
\subsection{Continuous family groupoid}
We equip $\cH(\cF)\ltimes \blup(\cF)$ with the product topology and $\hblup(\cF)$ with the quotient topology.
\begin{theorem}\label{main smoothness theorem}
\begin{enumerate}
\item Let $x\in M$, $\gamma\in \cH(\cF)_x$, $V\in \blup(\cF)_x$. Then the leaf $\cL_s(\gamma,V)$ is an embedded closed submanifold of $\cH(\cF)_x$.
\item The quotient space of $\cH(\cF)_x$ by the foliation given by $\cL_s$ is a smooth manifold whose tangent space at $x$ is $T\cF_{(V,x)}$.
\end{enumerate}
\end{theorem}
We remark that the quotient space is the same as $\hblup(\cF)_{(V,x)}$. We also have the parametrised version of Theorem \ref{main smoothness theorem}.
 \begin{theorem}\label{thm:closed leaves general}
Let $U$ be a bisubmersion, $u\in U$, $V\in\blup(\cF)_{s_U(u)}$, $S\subseteq U$ a smooth transversal of $\mathfrak{d}'s_U^{-1}(V)$ at $u$, in other words $$T_uS\oplus \mathfrak{d}'s_U^{-1}(V)_u=T_uU.$$ Then there exists $S',K$ open neighbourhoods of $u$ and $(V,x)$ in $S$ and $\blup(\cF)$ respectively such that the map $$\psi:S\times_{s_U,\pi_{\blup(\cF)}}K\to \hblup(\cF),\quad \psi(v,(W,s_U(v)))=\cQ(q_U(v),W)$$ satisfies the following 
\begin{enumerate}
\item $\psi$ is an open embedding.
\item the map $S'\cap U_y\to \hblup(\cF)_{(W,y)}$ which sends $v\in S'\cap U_y$ to $\cQ(q_U(v),W)$, is an open smooth embedding.
\end{enumerate}

\end{theorem}
\begin{proof}[Proof of Theorem \ref{main smoothness theorem} and Theorem \ref{thm:closed leaves general}.]
Let $k=\dim(\cF_x)$, $X_1,\cdots,X_k$, $r:M\times \R^k\to M $, $\cW\subseteq M\times \R^k$, as Example \ref{exs:bisubmersions}.2, $\phi_y:\cF_x\to \cF_y$ the linear map  defined in Eqn. \eqref{dfn phi}. Let $K$ be a open neighbourhood of $(V,x)\in\blup(\cF)$ with $\overline{K}$ compact and $K'=\pi_{\blup(\cF)}(\overline{K})\subseteq M$. We remark that $K'$ is compact but is not necessarily a neighbourhood of $x$, see Remark \ref{rem:open}.
\begin{lem}\label{lemma 2 gen}There exists $\eta>0$ such that for any $(y,t)\in K'\times \R^k$ if $r(y,t)=y$ and $\norm{t}<\eta$, then $\sum_{i=1}^k t_i X_i(y)=0$.
\end{lem}
The proof is similar to the proof of Lemma \ref{lemma 2} and will be omitted. After possibly reducing $\cW$, we can suppose that $\cW\subseteq M\times B(0,\eta)$, where $B(0,\eta
)$ is the ball of radius $\eta$. 
\begin{lem}\label{lem: af}If $(W,y)\in K$, $(y,t)\in \cW$, $q_\cW(y,t)\in \cL_s(y,W)$, then $\sum_{i=1}^k t_i[X_i]_y\in W$. Equivalently $\sum_{i=1}^k t_i[X_i]_x\in \phi_y^{-1}(W)$.
\end{lem}
\begin{proof}
By definition of $\blup(\cF)$, there exists $y_n\in M_{\reg}$ such that $(0,y_n)\to (W,y)$. Since $K$ is open, we can suppose that $(0,y_n)\in K$ for all $n$. Hence $y_n \in K'$. One can find $\gamma_n\in \cL_s(x_n,0)$ which converge to $q_\cW(y,t)$. This is done by covering a smooth path from $q_\cW(y,t)$ to $y$ in $\cL_s(y,W)$ with a finite number of bi-submersions, then using Lemma \ref{lem:continuity of leaves} on each.  By Remark \ref{rem:open map}.2, it follows that there exists $t_n$ such that $(y_n,t_n)\in \cW$ and $(y_n,t_n)\to (y,t)$ and $q_{\cW}(y_n,t_n)=\gamma_n$. It follows that $r(y_n,t_n)=y_n$. By Lemma \ref{lemma 2 gen}, $\sum_{i=1}^k t_{ni} [X_i]_{x}\in \phi_{y_n}^{-1}(\mathfrak{h}_{x_n})$. Hence  $\sum_{i=1}^k t_i[X_i]_x\in \phi_y^{-1}(W)$.
\end{proof}
Lemma \ref{lem: af} and Theorem \ref{thm:leaf decomp} applied to $(V,x)$ immediately imply that $\cL_s(x,V)$ is closed and that all leaves are closed embedded submanifolds. By \cite[Section 5.9.5]{BourbakiDiffVariety}, the quotient is a smooth manifold. This finishes the proof of Theorem \ref{main smoothness theorem}.

For Theorem \ref{thm:closed leaves general}. After possibly reducing $K,S$, by Lemma \ref{lem:continuity of leaves}, we can suppose that if $(W,y)\in \overline{K}$ and $v\in S\cap U_y$, then $$T_vS\oplus \mathfrak{d}'_vs_{U}^{-1}(W)=T_vU.$$
 By the definition of the topology on $\hblup(\cF)$, the map $\psi$ is continuous.  It is open by Remark \ref{rem:open map}.2 and the following lemma. \begin{lem}The quotient map $\cQ:\cH(\cF)\ltimes \blup(\cF)\to \hblup(\cF)$ is open.
\end{lem}
\begin{proof}
This follows immediately from Lemma \ref{lem:continuity of leaves} and Theorem \ref{thm:leaf decomp}.1.
\end{proof}


 We will show that we can reduce $S$, so that $\psi$ becomes an injective. Consider the bi-submersion $U^{-1}\circ U$. Since $q_{U^{-1}\circ U}(u,u)=x$, it follows that there exists $f:U^{-1}\circ U\to \cW$ a local morphism of bi-submersions such that $f(u,u)=(x,0)$. For $(W,y)\in \overline{K}$, consider the map \begin{align*}
 g_{(W,y)}:\{(v,v')\in \dom(f):v,v'\in S\}\to \frac{\cF_x}{\phi_y^{-1}(W)}\times U\\ g_{(W,y)}(v,v')=(\pi_{\cF_x}(f(v,v'))\mod  W,v'),
\end{align*} where $\pi_{\cF_x}:\cW\to \cF_x$ is the projection onto the second coordinate. The differential of $g_{(V,x)}$ at $(u,u)$ is injective. This is deduced from Lemma \ref{lem: af}. Since $\overline{K}$ is compact, we can find an open neighbourhood $S'\subseteq S$ of $u$ such that $g_{(W,y)}$ restricted to $\{(v,v')\in \dom(f):v,v'\in S'\}$ is injective.

We claim that $\psi$ restricted to $S'\times_{s_U,\pi_{\blup(\cF)}}K$ is injective. To see this let $v,v'\in S'$ such that $\cL_s(q_U(v),W)=\cL_s(q_U(v'),W)$. By Theorem \ref{thm:leaf decomp}, $$q_{\cW}(f(v,v'))=q_{U}(v')^{-1}q_{U}(v)\in \cL_s(y,W).$$ By Lemma \ref{lem: af}, $\pi_{\cF_x}(f(v,v'))\in \phi_y^{-1}(W)$. Hence $g_{(W,y)}(v,v')=(0,v')$. By the same argument $g_{(W,y)}(v',v')=(0,v')$. Injectivity claim follows.

Finally  the map $S'\cap U_y\to \hblup(\cF)_{(W,y)}$ which sends $v\in S'\cap U_y$ to $\cQ(q_U(v),W)$, is a smooth by Theorem \ref{main smoothness theorem} and Remark \ref{rem:open map}.1. By Remark \ref{rem:tangent space of hx}, we deduce that the derivative is injective at every point in the domain by transversality of $S$ and $\mathfrak{d}'s_{U}^{-1}(W)$.\end{proof}
\begin{cor}The groupoid $\hblup(\cF)$ is a not necessarily Hausdorff continuous family groupoid in the sense of \cite{ContFamilyGroupoids}
\end{cor}
The content of Theorem \ref{thm:closed leaves general} is that local charts in the sense of \cite{ContFamilyGroupoids} exist. We leave it to the reader to check that product and inverse on $\hblup(\cF)$ are $C^{\infty,0}$ smooth in the sense of \cite{ContFamilyGroupoids} (this follows from the fact that they both come from morphisms of bi-submersions which are $C^\infty$).

\begin{rems}[Smooth structure]\begin{enumerate}
\item Let $A\subseteq M$, a closed set, $N$ a smooth manifold, then we say $f:A\to N$ is a smooth if it admits a smooth extension to an open neighbourhood $U$ of $A$. Since $\blup(\cF)$ is locally a closed subset of $\Gr(\cF_x)\times M$, we can define $f:\blup(\cF)\to N$ to be smooth if its restriction to $\pi^{-1}(U)$ is smooth for any $X_1,\cdots,X_k$, $U\subseteq M$ as in Eqn. \eqref{eqn:top blowup}. One can easily check that different local charts of $\blup(\cF)$ are compatible, and so this notion of smoothness doesn't depend on the choice of $X_1,\cdots,X_k$, $U$. Furthermore by construction $ \pi:\blup(\cF)\to M$ is smooth and it is a diffeomorphism when restricted to $\pi^{-1}(M_{\reg})$.
\item If $\blup(\cF)$ is a smooth manifold, then the local charts in Theorem \ref{thm:closed leaves general} show that $\hblup(\cF)$ is a Lie groupoid.
\end{enumerate}
\end{rems}
\subsection{Integration}
Since $\hblup(\cF)$ is non Hausdorff in general, more care needs to be taken in defining its $C^*$-algebra. We refer the reader to \cite{SkandalisKohshkamReg,TimmermannFellComapct,TuNonHausdorff} for more details on the $C^*$-algebra of locally compact non Hausdorff groupoids.

Let $U$ be a bi-submersion. We define $\left(\Omega^\frac{1}{2}U\right)_u=|\Lambda|^\frac{1}{2}\ker(ds)_u\otimes |\Lambda|^\frac{1}{2}\ker(dr_U)_u$, see \cite[Section 4]{AS1}. Here $|\Lambda|^\frac{1}{2}$ denotes the bundle of $\frac{1}{2}$ densities. Let $u\in U,V\in \blup(\cF)_{s_U(u)}$. By proposition \ref{prop:image hx}, the following is an exact sequence $$0\to \mathfrak{d}'s_U^{-1}(V)_u\to \ker(dr_U)_u\xrightarrow{\mathfrak{d}s_U} \frac{\cF_{s_U(u)}}{V}\to 0.$$ It follows that we have a canonical isomorphism $$|\Lambda|^\frac{1}{2}\frac{\cF_{s_U(u)}}{V}\otimes |\Lambda|^\frac{1}{2}\mathfrak{d}'s_U^{-1}(V)_u\simeq |\Lambda|^\frac{1}{2}\ker(dr_U)_u.$$We also have the short exact sequence $$0\to \mathfrak{d}'s_U^{-1}(V)_u\to \ker(ds_U)_u\xrightarrow{\mathfrak{d}r_U} \frac{\cF_{r_U(u)}}{u\cdot V}\to 0.$$ Hence a canonical isomorphism $$|\Lambda|^\frac{1}{2} \frac{\cF_{r_U(u)}}{u\cdot V} \otimes |\Lambda|^\frac{1}{2}\mathfrak{d}'s_U^{-1}(V)_u\simeq |\Lambda|^\frac{1}{2}\ker(ds_U)_u.$$ Taking the tensor product of the two isomorphisms we get $$   |\Lambda|^\frac{1}{2}\frac{\cF_{s_U(u)}}{V} \otimes |\Lambda|^\frac{1}{2} \frac{\cF_{r_U(u)}}{u\cdot V}\otimes |\Lambda|^1\mathfrak{d}'s_U^{-1}(V)_u\simeq \left(\Omega^\frac{1}{2}U\right)_u.$$
Let $T\cF$ be the Lie algebroid of $\hblup(\cF)$. The last isomorphism can be thus rewritten as $$  |\Lambda|^\frac{1}{2}T\cF_{(V,s_U(u))}\otimes |\Lambda|^\frac{1}{2} T\cF_{(u\cdot V,r_U(u))} \otimes |\Lambda|^1\mathfrak{d}'s_U^{-1}(V)_u\simeq \left(\Omega^\frac{1}{2}U\right)_u.$$
We define $\Omega^\frac{1}{2} \hblup(\cF)$ to be the vector bundle \begin{align*}
 \Omega^\frac{1}{2} \hblup(\cF):&=|\Lambda|^\frac{1}{2}s_{\hblup(\cF)}^*T\cF\otimes |\Lambda|^\frac{1}{2}r_{\hblup(\cF)}^*T\cF
\end{align*}
Let $f\in \Gamma_c(U,\Omega^\frac{1}{2}U)$, $u\in U$. Then since the tangent bundle of $\cL_s(U,u,V)$ is equal to $|\Lambda|^1\mathfrak{d}'s_U^{-1}(V)$, we get that the following integral is well defined $$\int_{\cL_s(U,u,V)}f\in \left(\Omega^\frac{1}{2} \hblup(\cF)\right)_{(q_U(u),(V,x))}.$$
Let $\mathcal{A}$ be the algebra defined in \cite[Section 4.3]{AS1}. We have thus defined a map $$\fint:\cA\to \Gamma_c(\hblup(\cF),\Omega^\frac{1}{2}\hblup(\cF)).$$ It is straightforward to check that this map is a $*$-algebra homomorphism. Hence it passes to the completion $$\fint :C^*_{\max}\cF\to C^*_{\max}\hblup(\cF).$$
\section{Half-continuous fields of $C^*$-algebras and $z$-completion}\label{sec:semicont fields}
\subsection{Half-continuous fields} 
Let $X$ be a locally compact Hausdorff space, $A$ a $C_0(X)$-$C^*$-algebra in the sense of Kasparov \cite{KasparovInvent}, see also \cite{WilliamsCrossedBook,BlanchardHopf}. By \cite[Proposition C.5]{WilliamsCrossedBook}, $\hat{A}=\sqcup_{x\in X}\hat{A}_x$ (as a set not a topological space) and the projection map $\hat{A}\to X$ is continuous.
\begin{prop}\label{prop:semi cont}Let $x\in X$ be a non isolated point. The following conditions are equivalent \begin{enumerate}
\item for any $a\in A$, one has $$\limsup_{y\to x}\norm{a_y}=\norm{a_x}.$$
\item for any $a\in A$, if $a_y=0$ for all $y\in X\backslash \{x\}$, then $a=0$.
\item The set $\sqcup_{y\in X\backslash\{x\}}\hat{A}_y$ is dense in $\hat{A}$.
\end{enumerate}
\end{prop}
\begin{proof}
The implication $1\Rightarrow 2$ follows from  By \cite[Proposition C.10.a]{WilliamsCrossedBook}. For the converse, let $a\in A$. We can suppose that $a$ is self-adjoint by replacing $a$ with $a^*a$. Consider the $C^*$-subalgebra generated by $fa$ for $f\in C_0(X)$. This is a commutative $C_0(X)$-$C^*$-algebra. Hence it is isomorphic to $C_0(Y)$ for some locally compact space $Y$ equipped with a continuous map $\pi:Y\to X$. Part $2$ implies that if $f\in C_0(Y)$ is equal to $0$ on $\pi^{-1}(X\backslash \{x\})$, then $f=0$. Hence $\pi^{-1}(x)\subseteq \overline{\pi^{-1}(X\backslash \{x\})}$. Let $f\in C_0(X)$. Then $$\norm{f_{|\pi^{-1}(x)}}\leq \limsup_{y\to x}\norm{f_{|\pi^{-1}(y)}}.$$ The inequality $$\norm{f_{|\pi^{-1}(x)}}\geq \limsup_{y\to x}\norm{f_{|\pi^{-1}(y)}}$$ is straightforward to check. The equivalence of $2$ and $3$ is by the definition of the Jacobson topology on $\hat{A}$.
\end{proof}
We say that $A$ is half-continuous at $x$ if either $x$ is isolated or $A$ satisfies the conditions of Proposition \ref{prop:semi cont}. We say that $A$ is half-continuous if it is half-continuous at every point $x\in X$.
\begin{ex}The $C^*$-algebra $C([0,\frac{1}{2}])$ is a half-continuous $C([0,1])$-$C^*$-algebra but it is not a continuous $C([0,1])$-$C^*$-algebra.
\end{ex}
By Proposition \ref{prop:semi cont}, if $x_0\in X$ is non isolated and $I_{x_0}:=\{a\in A:a_y=0\forall y\in X\backslash \{x\}\}$, then $A/I_{x_0}$ is the maximal half-continuous at $x$ quotient of $A$. 
\paragraph{Fibered groupoids.} Let $G$ be a locally compact Hausdorff groupoid, $\lambda$ a continuous family of Haar measures on the $s$-fibers, see \cite[Def. I.2.2]{RenaultBook}. Let $p:G^0\to X$ be a continuous map such that $p\circ s=p\circ r$. It follows that for $x\in X$, $G_{p^{-1}(x)}=G_{p^{-1}(x)}^{p^{-1}(x)}$ is a subgroupoid of $G$. 
 
The $C^*$-algebras $C^*G$, $C^*_rG$ and $C_0(G^0)$ are $C_0(X)$-$C^*$-algebras. The following Proposition is an adaptation of \cite[Theorem 5.6]{LandsmanRamazanContinuousField} to half-continuous fields.
 \begin{theorem}\label{thm:cont field}Let $x\in X$.\begin{enumerate}
 \item The fiber at $x\in X$ of $C^*G$ is equal to $C^*G_{p^{-1}(x)}$.
 \item If $C_0(G^0)$ is half-continuous at $x$, then $C^*_rG$ is half-continuous at $x$.
 \item If $G_{p^{-1}(x)}$ is amenable, then the fiber at $x\in X$ of $C^*_rG$ is equal to $C^*_rG_{p^{-1}(x)}$.
 \item If $G_{p^{-1}(x)}$ is amenable and $C^*_rG$ is half-continuous at $x$, then $C^*G$ is half-continuous at $x$.
 \end{enumerate}
 \end{theorem}
 \begin{proof}
\begin{enumerate}
\item Consider the map $C_c(G)\to C^*G\to (C^*G)_x$, where the last map is the map taking the fiber of $C^*G$ at $x$. This map only depends on the restriction of $f\in C_c(G)$ to $C_c(G_{p^{-1}(x)})$. We thus get a map $C_c(G_{p^{-1}(x)})\to \left(C^*G\right)_x$. By definition of the maximal $C^*$-algebra, this extend to $*$-homomorphism $C^*G_{p^{-1}(x)}$, and we thus get \begin{align}
C^*G_{p^{-1}(x)}\to (C^*G)_x\label{eqn:max}
\end{align} It is surjective because the image was dense before extending to $C^*G_{p^{-1}(x)}$. The map \eqref{eqn:max} is injective by Renault's disintegration theorem \cite{RenaultBook}. 
\item 

Let $E$ be the $C_0(G^0)$-$C^*$-module which is obtained as the closure of $C_c(G)$ with respect to the inner product $$\langle f,g\rangle (y)=f^*\star g(y)=\int_{G_x} \bar{f}(\gamma)g(\gamma)d\lambda_x(\gamma),\quad y\in G^0.$$ Furthermore one has an injective $*$-homomorphism $C^*_rG\to \mathcal{L}(E)$.

Suppose $x$ isn't isolated. Let $a\in C^*_r(G)$ such that $a_y=0$ for all $y\in X\backslash\{x\}$. For $\xi\in E$, one has $$\langle a\star \xi,a\star \xi\rangle(y)=0,\quad \forall y\in p^{-1}(X\backslash \{x\}).$$ Since $C_0(G^0)$ is half-continuous at $x$, it follows that $\langle a\star \xi,a\star \xi\rangle=0$. Hence $a\star \xi =0$. Since this holds for all $\xi \in E$, it follows that $a=0$.
\item Since $G_{p^{-1}(x)}$ is amenable, by \cite{AnaRenaultGrps}, $C^*G_{p^{-1}(x)}=C^*_rG_{p^{-1}(x)}$. By \eqref{eqn:max}, one has \begin{align*}
C^*_rG_{p^{-1}(x)}=C^*G_{p^{-1}(x)}\to (C^*G)_x\to (C^*_rG)_x
\end{align*}
This map is surjective because it clearly has dense image. It is injective because the left regular representation on $L^2G_y$ for $y\in p^{-1}(x)$ factors through $(C^*_rG)_x$.
\item Let $a\in C_c(G)$. By the previous parts, one has \begin{align*}
 \norm{a_{|G_{p^{-1}(x)}}}_{C^*G_{p^{-1}(x)}}=\norm{a_x}_{(C^*_rG)_x}=\limsup_{y\to x}\norm{a_y}_{(C^*_rG)_y}&\leq \limsup_{y\to x}\norm{a_y}_{(C^*G)_y}\\ &\leq \norm{a_x}_{(C^*G)_x}=\norm{a_x}_{C^*G_{p^{-1}(x)}},
\end{align*}
where we used \cite[Proposition C.10.a]{WilliamsCrossedBook} in the last inequality. The result follows.\qedhere
\end{enumerate}
\end{proof}
\begin{rems}
\begin{enumerate}
\item In general the fiber of $C^*_rG$ at $x$ is a quotient of $C^*G_{p^{-1}(x)}$ but not equal to $C^*_rG_{p^{-1}(x)}$. For example if $G^0=X=\{\frac{1}{n}:n\in \N\}\sqcup\{0\}$ with $G_{\frac{1}{n}}=SL(2,\frac{\Z}{n\Z})$ and $G_{0}=SL(2,\Z)$, see \cite[Section 2]{BCcounter}. 
\item Theorem \ref{thm:cont field} is also true for locally compact locally Hausdorff groupoids with $G^0$ Hausdorff. Part 1, one uses Renault's disintegration theorem for non Hausdorff groupoids \cite[Theorem 7.8]{MuhlyWilliamsRenaultTheoremsNonHaus}. In part 2, one replaces the module $E$ with the module constructed in \cite{SkandalisKohshkamReg}. In this case, one constructs a $C_0(Y)$-$C^*$-module $E$ with the proeprty that $C^*_rG\to \cL(E)$ is injective. The space $Y$ is constructed in \cite[the paragraph preceding Proposition 2.6]{SkandalisKohshkamReg}. Semi-continuity of $C_0(G^0)$ over $X$ together with \cite[Proposition 2.6]{SkandalisKohshkamReg} imply semi-continuity of $C_0(Y)$ over $X$.
\end{enumerate} 
\end{rems}
\paragraph{Fibered foliations.}Let $M,Y$ be smooth manifolds, $p:M\to Y$ a surjective submersion, $\cF\subseteq \Gamma(\ker(dp))$ a singular foliation. The space $\blup(\cF)$ is fibered over $Y$ through the map $p\circ \pi_{\blup(\cF)}$. Since $\cF$ is longitudinal to $\ker(p)$, $\cH(\cF)$ and $\hblup(\cF)$ are fibered over $Y$.\footnote{Here we make use of the assumption that we only consider path holonomy bi-submersions. For a path holonomy bi-submersion $U$, one has $p\circ s_U=p\circ r_U$.} \begin{cor}\label{cor:semicont}The $C^*$-algebra $C^*_r\hblup(\cF)$ is a half-continuous $C_0(Y)$-$C^*$-algebra. \end{cor}
\begin{proof}
By construction of $\blup(\cF)$, the set $\pi^{-1}_{\blup(\cF)}(M_{\reg})$ is dense in $\blup(\cF)$. Since $p$ is a surjective submersion, $p(M_{\reg})$ is an open dense subset of $Y$. Semi-continuity of $C_0(\blup(\cF))$ follows immediately from Proposition \ref{prop:semi cont}. Semi-continuity of $C^*_r\hblup(\cF)$ follows from Theorem \ref{thm:cont field}.2.
\end{proof}
\subsection{Characteristic set and $z$-completion}
In this section, we define the $z$-completion and the characteristic set of a singular foliation. We prove a theorem which is used by the author in various applications, whose proof relies on the blowup construction.
\begin{dfn}Let $\cF$ be a singular foliation. We define\begin{itemize}
 \item $C^*_z\cF:=\frac{C^*\cF}{\ker(\fint)}$.
\item for $x\in M$, the characteristic set of $\cF$ at $x$ denoted $\cT^*\cF_x$ as the set $\bigsqcup_{V\in \blup(\cF)_x}V^\perp\subseteq \cF^*_x$. We also define $\cT^*\cF:=\bigsqcup_{x\in M}\cT^*\cF_x$ with the subspace topology from the inclusion $\cT^*\cF\subseteq \cF^*$.
\end{itemize}
\end{dfn}
The map $\pi_{\blup(\cF)}:\blup(\cF)\to M$ being proper implies that \begin{itemize}
\item Both $\cT^*\cF_x$ and $\cT^*\cF$ are closed subsets of $\cF^*_x$ and $\cF^*$ respectively.
\item  $\cT^*\cF=\overline{\bigsqcup_{x\in M_{\reg}}\cF_x^*}$.
\end{itemize} 
\begin{theorem}Let $Y$ be a smooth manifold, $p:M\to Y$, a surjective submersion, $\cF\subseteq  \Gamma(\ker(dp))$ a singular foliation. We suppose 
\begin{enumerate}
\item $p^{-1}(p(M_{\reg}))=M_{\reg}$.
\item for every $x\in M_{\reg}^c$, $\mathfrak{h}_x=\cF_x$
\item for every $x\in M_{\reg}^c$, the simply connected Lie group integrating $\mathfrak{h}_x$ is amenable. 
\item $C^*\cF$ is half-continuous at every $y\in p(M_{\reg})$.
\end{enumerate}
Then \begin{enumerate}
\item $C^*_z\cF=C^*\cF/I,\quad I=\{a\in C^*\cF:a_y=0\quad\forall y\in p(M_{\reg})\}$.
\item $C^*_z\cF$ is a half-continuous at every $y\in Y$.
\item $\cH(\cF)^x_x$ is a simply connected Lie group integrating $\mathfrak{h}_x$.
\item $C^*_z\cF$ lies in a short exact sequence $$0\to C^*\cF_{M_{\reg}}\to C^*_z\cF\to A\to 0,$$ where $A$ is a $C^*$-algebra fibered over $M_{\reg}^c$. Its fiber at $x\in M_{\reg}^c$ is equal to $C^*\cH(\cF)_x^x/J_x$, where $$J_x=\bigcap_{V\in \blup(\cF)_x}\ker(\pi_V),\quad \pi_V:C^*\cH(\cF)_x^x\to B(L^2(\cH(\cF)_x^x/V)).$$ Here $\pi_V$ is the quasi-regular representation.
\end{enumerate}
\end{theorem}
\begin{proof}
\begin{enumerate}
\item The equality \begin{equation}\label{eqn:equ hblup h reg}
\hblup(\cF)_{M_{\reg}}=\cH(\cF)_{M_{\reg}}
\end{equation} together with Assumption 1 imply that $C^*\hblup(\cF)$ is half-continuous at $y\in p(M_{\reg})$. If $y\in p(M_{\reg}^c)$, then Assumption 3, Theorem \ref{thm:cont field}.4, Corollary \ref{cor:semicont} imply that $C^*\hblup(\cF)$ is half-continuous at $y$. Hence $C^*\hblup(\cF)$ is half-continuous at every point. Since the map $\fint$ respects the fibration over $X$, we deduce that $\fint(I)=0$. Hence $I\subseteq \ker(\fint)$. By Eqn. \eqref{eqn:equ hblup h reg}, it follows that $\ker(\fint)\subseteq I$. The equality $C^*_z\cF=C^*\cF/I$ follows.
\item Since $C^*_z\cF$ is a $C^*$-subalgebra of $\hblup(\cF)$, half-continuity follows.
\item This follows from Assumption $2$, Theorem \ref{thm:AZ} and the paragraph preceding it.
\item The short exact sequence exists by Equation \eqref{eqn:equ hblup h reg}. By Assumption 2, $A$ is fibered over $M_{\reg}^c$. Let $A_y$ be the fiber of $A$ over $y\in p(M_{\reg}^c)$. It is the union of all fibers $A_x$ for $x\in p^{-1}(y)$. The map $\fint$ induces a map $\fint:A_y\to C^*\hblup(\cF)_y$. This map is injective because $\fint:C^*_z\cF\to C^*\hblup(\cF)$ is the identity on the regular part and both $C^*_z\cF$ and $C^*\hblup(\cF)$ are half-continuous. Since $\hblup(\cF)_y$ is amenable, by Theorem \ref{thm:cont field}.1, its fiber over $x\in p^{-1}(y)$ is $\hblup(\cF)_x$. Hence $\fint:A_x\to C^*\hblup(\cF)_x=C^*_r\hblup(\cF)_x$ is injective. The result follows because the left regular representations of $C^*\hblup(\cF)$ at $V\in \blup(\cF)_x$ composed with $\fint:C^*\cH(\cF)_x^x\to C^*\hblup(\cF)$ are the quasi regular representations $\pi_V$.\qedhere
\end{enumerate}
\end{proof}
\begin{ex} If $\cH(\cF)_x^x$ is commutative, then $C^*\cH(\cF)_x^x=C_0(\cF_x^*)$ and $A_x=C_0(\cT^*\cF_x)$. More generally if $G=\cH(\cF)_x^x$ is nilpotent, then by the orbit method \cite{KirillovArticle,BrownArticleTopOrbitMethod}, we can identify the spectrum of $C^*G$ with $\mathfrak{h}_x^*/Ad^*(G)$. By \cite[Lemma 2]{LudwigGoodidealart}, the support of $\pi_V$ is equal to $\overline{\bigcup_{g\in G}\mathrm{Ad}^*(g)V^\perp}/Ad^*(G)\subseteq \mathfrak{h}_x^*/Ad^*(G)$. Since the set $\blup(\cF)_x$ is closed under the Adjoint action of $G$ by Theorem \ref{thm:funct blowup} (applied to $\exp(X)$ for $X\in \cF$). It follows that the quotient $A_x$ is the quotient of $C^*G_x$ which corresponds to the closed set $\cT^*\cF_x/Ad^*(G)$. 
\end{ex}
\begin{refcontext}[sorting=nyt]
\printbibliography
\end{refcontext}
{\footnotesize
		 (Omar Mohsen) Paris-Saclay University, Paris, France
		\vskip-10pt e-mail: \texttt{omar.mohsen@universite-paris-saclay.fr}}
\end{document}